\documentclass[12pt,oneside,pb-diagram]{article}
\newcommand{\Path}{}
\newcommand{\figs}{}
\usepackage{ \Path Paper_style_2}
\usepackage{ \Path Keywords_WB}
\usepackage{ \Path Category}
\usepackage{ \Path Environments}
\usepackage{tikz-cd}
\usepackage{mathtools}
\usepackage{\Path tikzcd2}
\usepackage{pb-diagram}
\usepackage{\Path Paper_keywords}
\renewcommand{\eval}{ \text{eval} }
\renewcommand{\Transfer}[1]{ \text{Tr}_{#1} }

\title{Dynamical systems as enriched functors}

\begin{document}
	\author{ Suddhasattwa Das \footnotemark[1] \footnotetext[4], Tomoharu Suda \footnotemark[2] \footnotemark[3]}
	\footnotetext[1]{Department of Mathematics and Statistics, Texas Tech University, USA}
	\footnotetext[2]{Department of Applied Mathematics, Tokyo University of Science, Japan}
	\footnotetext[3]{RIKEN Center for Sustainable Resource Science, RIKEN, Japan}
	\footnotetext[3]{Corresponding author}
	\date{\today}
	\maketitle
	\begin{abstract}
		This article presents a general description of dynamical systems using the language of enriched functors and enriched natural transformations. This framework is essential to establish the equivalence of three descriptions of dynamics -- a semigroup action on the domain; a parameterized family of endomorphisms; and a transformation of time-space into the collection of endomorphisms. A collection of categorical axioms are presented that provides a complete categorical language to develop dynamical systems theory. None of the assumptions are rooted in specific contexts such as topology and measure spaces. The equivalence of the three descriptions is further used to construct other related notions,such as transfer operators, orbits and shift-spaces. All of these objects are defined by their structural role and universal properties, instead of their usual pointwise definitions.
	\end{abstract}
	
	
	%-_-_-_-_-_-_-_-_-_-_-_-_-_-_-_-_-_-_-_-_-_-_-_-_-_-_-_-_-_-_-_-_-_-_-_-_-_-_-_-_-_-_-_-_-_-_-_-_-_-_-_-_-_-_-_-_-_-_-_-_-_-_-_-_-_-_-_-_-_-_-_-_-_-_-_-
	\section{Introduction} \label{sec:intro}
	
	A dynamical system consists of a collection $\Omega$ of states, and a law describing a repeated transition between states along with the progress of time. The collection $\Omega$ may be endowed with a sigma-algebra, a linear structure, a differential structure or a topology. If the dynamics law also conforms to these structures, e.g., it is measurable, linear, differentiable or continuous respectively, then one can study the dynamics in these different contexts. These contexts, namely ergodic, linear, differential and continuous dynamical systems have their own theories and constructions. The classical approach to dynamical systems theory is to discover the local structures within the state space, emerging as a consequence of the dynamics law.
	
	\begin{figure}[!t]
		\centering
		\begin{equation} \label{eqn:thm:8}
			\begin{tikzcd}
				&& \begin{array}{c} \text{Parameterized} \\ \text{endomorphisms} \\ \ParaEndoCat{\Context} \end{array} \arrow[drr, "\text{flow}", "\cong"'] \\
				\begin{array}{c} \text{Enriched} \\ \text{functors} \\ \Time \to \Context \\ \EnrichSGActCat{\Context} \end{array} \arrow[rru, "\text{parametric}", "\cong"'] \arrow[rrrr, "\cong"'] && && \begin{array}{c} \text{Left-} \\ \text{actions} \\ \FlowCat{\Context} \end{array}
			\end{tikzcd}
		\end{equation}
		\caption{\textbf{Three functorial descriptions of dynamics.} One of the main results of the paper, formally stated in Theorem \ref{thm:8}, has been depicted here as a three way equivalence of categories.}
		\label{fig:thm:8} 
	\end{figure}
	
	An alternative to the classic approach is a categorical approach \cite[e.g.]{BehrischEtAl2017dyn, fritz2020Stoch, MossPerrone2022ergdc, DasSuda2024recon, Suda2022Poincare, SchultzEtAl2020dyn, VagnerSpivakLerman2015, LibkindEtAl2021operadic}, which takes a global view of dynamical systems. The narrative then takes a structural or synthetic form, in which various properties such as invariance, ergodicity and conjugacy have simpler and more intuitive definitions in terms of their structural roles, instead of being a consequence of their usual pointwise definitions. The synthetic description of category theory often makes several statements and proofs shorter, almost tautological. The complexity of the proof is transmitted into the challenge of establishing borne by the categorical structures. 
	
	One of the standard categorical descriptions of dynamics is as a functor $\Phi$ from a semigroup into a context category. This provides the most direct and concise representation as a categorical object. We discuss later in diagram \eqref{eqn:chain:3} that this simple idea allows the collection of linear, differentiable, topological and measurable dynamical systems, and Markov processes to become categories of its own, with a nesting of these categories in the order presented. Thus the various arrangements of dynamical systems created by different contexts and different time-structure are all categories, and are inter-related functorially. 
	
	The description of a dynamical system as a functor encodes its algebraic property of being the action of a semi-group. However, this description is not adequate to encode the many other descriptions of a dynamical system, each suited for a specific purpose. A dynamical law may be presented as a map $\Phi : \TimeB \times \Omega \to \Omega$, to directly indicate the dependence of the future state on the present time and state. Here $\TimeB$ represents some structure of time. These two interpretations of a dynamical system have lead to mostly separate lines of development of categorical dynamics, such as the former approach in \cite[e.g.]{MossPerrone2022ergdc, DasSuda2024recon}, and the latter as open dynamical systems \cite[e.g.]{VagnerSpivakLerman2015, FongEtAl2016open}. There is yet a third representation of the dynamics as a parameterized collection of self-maps $\Phi^\sharp : \TimeB \to \Endo(\Omega)$, where $\Endo(\Omega)$ is the set of self-maps on $\Omega$. In this case $\Phi^\sharp$ does not encode the algebra, but the nature of the dependence of the transformation on time. 
	
	The vast body of theory in dynamical systems depends on our ability to use these descriptions interchangeably. It is easy to follow the equivalence of these descriptions within a set-theoretic context. However, the interchangeability is highly non-trivial in a general context, and has been missing in most categorical descriptions of dynamics. Our main goal is to establish the axiomatic roots of this interchangeability. Such an axiomatization reveals the logical dependencies of various dynamic notions such as orbits, stationary states and transfer operators. Each such dependency would be a consequence of structural relations, rather than properties defined pointwise.
	
	Our approach is distinguished by the utilization of enriched functors and enriched natural transformations. We present a set of purely category theoretic assumptions that established equivalence between all the usual descriptions of dynamics. Due to our commitment to a structural or categorical description, no specific properties such as continuity, differentiability or measurability are assumed. Instead all properties are inherited from the categorical structure of the context, to be denoted as $\Context$. One of our main results is 
	
	\begin{corollary} \label{corr:2}
		Under the categorical assumptions \ref{A:dyn}, \ref{A:monoidal}, \ref{A:closed} and \ref{A:enrich_fnctr}, the three descriptions of a dynamical system - enriched functor, parameterized family, and left action of a semi-group, are equivalent, as displayed in Figure \ref{fig:thm:8}.
	\end{corollary}
	
	Corollary \ref{corr:2} is a consequence of the more rigorous statement of Theorem \ref{thm:8}. All our results are built upon four basic categorical assumptions. Thus we present a synthetic approach to dynamical systems. The building blocks are not points and sets, but transitions and their inter-relations. Our axiomatic approach is outlined in Figure \ref{fig:dyn_interpret}. Figure \ref{fig:dyn_interpret} also displays four classical interpretations of dynamical systems. An important consequence of our axiomatization is
	
	\begin{corollary} \label{corr:5}
		Under the categorical assumptions \ref{A:dyn}, \ref{A:monoidal}, \ref{A:closed} and \ref{A:enrich_fnctr} on a context category $\Context$, one can associate to every object $\Omega\in \Context$ the notions of a path-space, orbit-space, and the collection of endomorphisms of $\Omega$. Each of these notions emerge as new objects in $\Context$.
	\end{corollary}
	
	While these notions are familiar and intuitive in a set-theoretic approach, they are rediscovered solely from categorical assumptions in this article. Each of the three descriptions have been investigated separately, for example as action of monads \cite[e.g.]{BehrischEtAl2017dyn}, as functors \cite[e.g.]{DasSuda2024recon, MossPerrone2022ergdc}, and as parametrized action of monads \cite[e.g.]{pareigis1977moduleA, pareigis1977moduleB}. These constructions are defined in Section \ref{sec:enrich_fnctr} and inter-related by leveraging the power of \emph{enriched categories}. This groundwork enables a host of new constructions :
	
	\begin{corollary} \label{corr:4}
		Suppose that the contexts of set theory, topology, measurable spaces and smooth spaces are represented by the closed monoidal categories $\SetCat$, $\CGWHS$, $\QuasiBorel$ and $\Diffeological$ respectively. Given a dynamical system $\paran{\Phi^t, \Omega}$ in any of the contexts, one can associate to it two new dynamics within the same context :
		\begin{enumerate} [(i)]
			\item the pointwise action $\Transfer{\Phi}^t$ of the dynamics on curves
			\item the sub-flow $\shift^t$ induced on the path space of the domain.
		\end{enumerate}
		Both of these associations are functorial.
	\end{corollary}
	
	The two new induced dynamics -- $\Transfer{\Phi}^t$ and $\shift^t$ are in the same context as $\Phi$ in the sense that the spaces and transformations have the same type as allowed by the context $\Context$. The categories $\CGWHS$, $\QuasiBorel$ and $\Diffeological$ are discussed in Section \ref{sec:conclus}. They provide the proper collection of transformations that satisfy the axioms mentioned above. The dynamics of $\Transfer{\Phi}^t$ and $\shift^t$ are well familiar in classical dynamical systems theory as the transfer operator and shift operators on the space of paths. These are usually defined by their pointwise actions. Corollary \ref{corr:4} presents them as a purely categorical consequence. The details are presented in Theorems \ref{thm:1} and \ref{thm:2} respectively. These dynamics could be defined in any setting which lack an apriori notion of points or states. 
	
	A major gain from the perspective of enriched categories is a refined notion of orbits. An orbit of a point is the succession of states that it transitions to. The phase space is a disjoint collection of orbits, and the complexities of the dynamics lie in their mutual arrangement. Orbits have even been suggested as a primitive object in the description of dynamics \cite[e.g.]{Suda2023dynamical, Suda2023partial}. Another main gain from this perspective of enriched category is the notion of the space of paths $\PathF(\Omega)$, a natural dynamical system that can be ascribed to every object $\Omega$ in $\Context$. A deeper inspection of its construction reveals that any dynamical system remains embedded within the path-space as a \textit{subshift} :
	
	\begin{corollary} \label{corr:subshift}
		One can associate to a dynamical system $\paran{\Phi^t, \Omega}$ in any of the contexts in Corollary \ref{corr:4}, a subshift dynamics $\paran{\shift_{\Phi}^t, E_{\Phi}}$ within the same context. The subshift is defined via the following universal construction --
		\begin{enumerate} [(i)]
			\item the domain is the maximal collection of states which generate the same orbits under $\Transfer{\Phi}^t$ and $\shift^t$.
			\item the dynamics under $\shift_{\Phi}^t$ commutes with both $\Transfer{\Phi}^t$ and $\shift^t$.
		\end{enumerate}
	\end{corollary}
	
	Corollary \ref{corr:subshift} is stated in more technical terms in Theorem \ref{thm:3}. Each point in a subshift is the forward orbit of some point in $\Omega$. We further go on to prove that a subshift is precisely an equivalent description of the original dynamics, in terms of its orbits :
	
	\begin{corollary} \label{corr:dom_subshift}
		Suppose that the time semigroup is commutative. Then the domain of the subshift is an isomorphic copy of the domain $\Omega$ which is embedded within the path space $\PathF (\Omega)$.
	\end{corollary}
	
	Corollary \ref{corr:dom_subshift} will be derived as an immediate consequence of Theorems \ref{thm:3} and \ref{thm:6}. These constructions are introduced in Section \ref{sec:path}, and their properties are stated rigorously later in Theorem \ref{thm:1} and Theorem \ref{thm:2}. 
	
	Since our approach is synthetic / constructivist instead of analytical / descriptive, these results are proven without relying on details specific to topology, measure theory and stochastic processes. To introduce a categorical perspective into dynamical systems theory, or any other branch of mathematics, it is necessary to reframe that discipline. As a result, one must rediscover many foundational rules, properties, or theorems. Although repetition might be observed in the statements of these theorems, their underlying premises and methods of proof are entirely distinct and new.
	
	The utility of enriched categories was first realized by Lawvere \cite{Lawvere1973metric} when generalized metric spaces were re-interpreted as categories enriched over $[0, \infty]$. This lead to a comprehensive mathematical program designed to unify quantitative measurement, generalized logic, and continuum mechanics. By replacing the enriching poset with an arbitrary closed monoidal category $\mathcal{V}$ serving as a system of generalized truth values, Lawvere extended this enriched framework far beyond geometric distances to encapsulate deep structural relationships in topological spaces, probability theory, and thermodynamics \cite{Lawvere1986serious}. Subsequently, enriched categories have been used with success in several other fields, such as Algebraic topology \cite[e.g.]{Keller2007graded, Bergner2007model} and the idea of weighted limits \cite[e.g.]{day1970closed}. The present work is in the same spirit. Associated ideas, such as uniqueness of enrichment \cite{power2011unicity} and strength of monads \cite{McdermottUustalu2022makes} might provide further tools towards developing a coordinate-free foundation for studying dynamics.  A different but interesting perspective was offered by Myers \cite{Myers2021ecds} which interprets an enriched category itself as a dynamical system.
	
	\paragraph{Outline} We begin with a brief review of category theory in Section \ref{sec:cat}. We introduce the first axiom that enables a categorical description of dynamical systems. Time is assumed to be a semigroup from the outset. Our next axioms in Section \ref{sec:time} imparts a finer structure to the time semigroup. They also enable self-maps and paths to be interpretable as new objects within the same category. Our fourth and final axiom is presented in Section \ref{sec:enrich_fnctr}. It provides the connection between all the three descriptions of dynamical systems. These four axioms enable a deeper understanding of constructs such as path spaces and transfer operator. We analyze these in Section \ref{sec:path}. Next in Section \ref{sec:states} we incorporate the study of states and stationary states, which generalize the notion of both invariant sets and measures. We end with some applications and discussion in Section \ref{sec:conclus}.

	%-_-_-_-_-_-_-_-_-_-_-_-_-_-_-_-_-_-_-_-_-_-_-_-_-_-_-_-_-_-_-_-_-_-_-_-_-_-_-_-_-_-_-_-_-_-_-_-_-_-_-_-_-_-_-_-_-_-_-_-_-_-_-_-_-_-_-_-_-_-_-_-_-_-_-_-
	\section{Dynamics as a functor} \label{sec:cat}
	
	\begin{table} [t!]
		\caption{Various kinds of dynamical systems. Section \ref{sec:cat} presents how the collection of dynamical with a given time-structure $\calT$ and mathematical context $\calC$ has the structure of a functor-category. The table presents a few instances of $\Time$ and $\Context$. Each leads to a different sub-branch in the theory of dynamical systems. A categorical approach unifies many observations in these separate branches.  }
		\begin{tabularx}{\linewidth}{|l|L||l|L|}
        $\calG$ & Interpretation & $\calC$ & Interpretation \\ \hline \hline
        $\real$ & Invertible flow, continuous time & $\Topo$ & Topological dynamics \\ \hline
        $\Rplus$ & Continuous time & $\MeasCat$ & Measurable dynamics \\ \hline
        $\num$ & Discrete time & $\BanachCat$ & Operator family \\ \hline
        $\integer$ & Invertible, Discrete time & $\StochCat$ & Markov process \\ \hline
        $\mathbb{T}^d$ & Torus actions & $\SetCat$ & Set-theoretic dynamics \\ \hline
    \end{tabularx}
		\label{tab:TC}
	\end{table}
	
	As mentioned before, every notion of a dynamical system has two ingredients - a collection of \emph{states}, called the \emph{domain} or \emph{phase-space}, and a collection of transformations of the domain. The simplest example is when the domain $\Omega$ is a set, i.e., an unstructured collection of states, and the dynamics is a collection of maps from $\Omega$ to itself. This is a set-theoretic dynamical system, and is listed along with several other kinds of dynamics in Table \ref{tab:TC}. Dynamical systems theory usually assumes some structure on $\Omega$, such as measurability, probability, linearity or topology, and expects that the transformations respect these properties. For the examples listed, this means that the transformations are measurable, measure-preserving, linear and continuous respectively. One also assumes a certain algebraic structure on the transformations, which allows the transformations to be applied repetitively. Each of the collection of structure-preserving self-transformations are called \emph{endomorphisms}. Endomorphisms have a semigroup structure to them. The term ``endomorphism(s)" always refers to such a semigroup created by some structural aspect of the domain $\Omega$. A dynamical system is tied to the semigroup property of endomorphisms.
	
	\paragraph{First description} A dynamical system is the action $\tilde{\Phi}$ of a semigroup $\Time$ on the domain $\Omega$. Thus $\tilde{\Phi}$ is homomorphism  from $\Time$ to the semigroup of endomorphisms of $\Endo(\Omega)$.
	The time-semigroup $\Time$ could be the additive semigroups $\SqBrack{\real_0, +}$, $\SqBrack{\real, +}$, $\SqBrack{\num_0, +}$ and $\SqBrack{\num, +}$, and even rotation groups such as the torus group $\mathbb{T}^d$ \cite[e.g.]{raissy2010torus, Das2023Koop_susp}. One of the main considerations of  dynamical systems theory is that these systems co-exist, often in a nested manner. Most of the complicated behavior in dynamical systems theory is an outcome of the interplay between various co-existing sub-systems \cite[e.g.]{DSSY_Mes_QuasiP_2016, DasJim2017chaos, AlexanderEtAl1992, newhouse2004new, LiaoEtAl2018}. The description of individual dynamical systems along with their inter-connections, become more concise when one uses the language of \emph{category theory}.
	
	\paragraph{Category} A category $\calC$ is a collection of two kinds of entities : 
	\begin{enumerate} [(i)]
		\item objects : usually representing different instances of the same mathematical construct;
		\item morphism : connecting arrows between a pair of objects equipped with compositionality : given any three objects $a,b,c$ of $\calC$ and two morphisms $a \xrightarrow{f} b$ and $b \xrightarrow{g} c$, the morphisms can be joined end-to-end to create a composite morphism represented as $a \xrightarrow{g \circ f} b$. This composition is also associative;
		\item identity morphism : each object $a$ is endowed with a morphism $\Id_a$ called the \emph{identity} morphism, which play the role of unit element in composition.
	\end{enumerate}
	
	Given two objects $x,y$ in $\calC$, the collection of arrows from $x$ to $y$ is denoted as $\Hom_{\calC}(x;y)$. If the category $\calC$ is obvious, then it will be dropped from the notation and $\Hom(x;y)$ will be used for simplicity. Note that this collection may be infinite, finite or even empty. The last criterion implies that for each $x$ $\Hom(x;x)$ has at least one member. Whenever there are multiple categories being discussed, one uses the notations $\Hom_{\calC}(x;y)$ or $\calC(x;y)$ to indicate that the morphisms are within the category $\calC$.
	
	\paragraph{Examples} One of the most fundamental categories is $\SetCat$, the category in which objects are sets up to a certain prefixed cardinality, and arrows are arbitrary maps. Similarly $\Topo$ denotes the category of topological spaces, with continuous maps as arrows. We denote by $\VectCat$ the category in which the objects are vector spaces and arrows are linear maps. The collection $\AffineCat$ that we have already defined has the same objects as $\VectCat$ but all affine maps as morphisms. Note that this includes the morphisms in $\VectCat$. This makes $\VectCat$ a \emph{subcategory} of $\AffineCat$. Suppose $\calU$ is any set. Then the power-set $2^{\calU}$ of subsets of $\calU$ is a category, in which the relations are the subset $\subseteq$ relations. Note that there can be only at most one arrow between any two objects $A,B$ of this category, which is to be interpreted as inclusion. Such categories are known as \emph{preorders}, and other examples are the category of ordered natural numbers, real numbers, open covers, and the concept of infinitesimal \cite[see]{Das2023CatEntropy}. 
	
	A particularly important example of categorical structure can be found in semigroups. Any semigroup $\calG$ can be interpreted as a 1-point category, in which each arrow corresponds bijectively to elements of $\calG$. The associativity of semigroup operation becomes the associativity of arrow composition. Figure \ref{fig:semigroups} show  the 1-point categorical representations of the semigroups $\SqBrack{\num_0, +}$, $\SqBrack{\integer, +}$ and $\SqBrack{\real, +}$. We formalize this into an assumption :
	
	\begin{Assumption} \label{A:dyn}
		There is a semigroup $\Time$, interpreted as a 1-object category.
	\end{Assumption}
	
	\begin{figure} [!t]
		\center
		\includegraphics[width=0.26\textwidth]{\figs 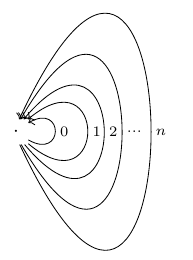}
		\includegraphics[width=0.47\textwidth]{\figs 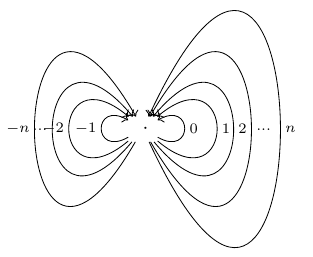}
		\includegraphics[width=0.24\textwidth]{\figs 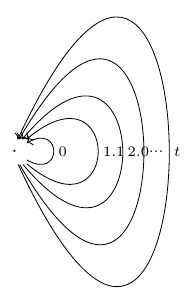}
		\caption{Representing semigroups as 1-object categories -- $\SqBrack{ \num_0, +}$, $\SqBrack{ \integer, +}$ and $\SqBrack{ \real, +}$.} 
		\label{fig:semigroups}
	\end{figure}
	
	Any dynamical system has a concept of time associated to it. The role of time can be played by any semigroup, such as $(\num_0, +)$, $(\num, +)$, $(\real_0, +)$ and $(\real, +)$. Such a time semigroup shall be denoted by $\Time$. We shall interpret it as a 1-point category. Next let us fix a category $\Context$ and call it the \emph{context} category. It is usually taken to be $\Topo$, $\VectCat$ or $\MeasCat$. The category $\Context$ represents the nature of spaces and maps that are involved our dynamical systems of interest. The context and time are the two ingredients of a dynamical systems. 
	
	\paragraph{Functors} Given two categories $\calC, \calD$, a functor $F:\calC\to \calD$ is a mapping between their objects along with the following properties : 
	\begin{enumerate} [(i)]
		\item For each $x,y\in ob(\calC)$, $F$ induces a map $F_{x,y} : \Hom_{\calC}(x;y) \to \Hom_{\calD}( Fx; Fy)$. Thus arrows / morphisms between any pair of points get mapped into arrows between the corresponding pair of points in the image.
		\item $F$ preserves compositionality : given any three objects $a,b,c$ of $\calC$ and two morphisms $a \xrightarrow{f} b$ and $b \xrightarrow{g} c$, $F(g\circ f) = F(g) \circ F(f)$.
		\item $F$ preserves identity : $F(\Id_a) = \Id_{F(a)}$.
	\end{enumerate}
	
	Functors are thus maps that also preserve categorical structure, or in other words, preserves compositionality. One of the simplest notions of functors are monotonic functions from $\real$ to $\real$, and may be interpreted as functors from $\Rplus$ to $\Rplus$ or $\Rminus$. Realizing mathematical transformations as functors leads to deeper insights in the associated field. Some examples are Lebesgue integration \cite{Leinster2020integration}, diameter and Lebesgue number \cite{Das2023CatEntropy}, and topological closure \cite{ClementinoTholen1997sep, ClementinoGiuliTholen1996}. This prompts our first interpretation : 
	
	\begin{definition} [Interpretation 1 - Semigroup action] \label{desc:dyn:1}
		Under Assumption \ref{A:dyn}, a dynamical system within a context category $\Context$ is a functor $\tilde{\Phi} : \Time \to \Context$.
	\end{definition}
	
	Any functor from a 1-object category such as $\Time$ has a single object in its image. The choice of this object corresponds to the choice of the domain $\Omega$ of the dynamics. Next involved is a family of endomorphisms $\Phi^t : \Omega \to \Omega$ indexed by $\braces{ t\in \Time }$. The compositionality of a functor is precisely an action of the semigroup $\calT$ on endomorphisms of $\Omega$. See Figure \ref{fig:dyn_functor} for an illustration. As mentioned before, the relations between dynamical systems are equally important as the dynamics within individual systems. These relations can be expressed concisely in categorical language as a transformation of functors.
	
	\paragraph{Natural transformations} Let $\calC, \calD$ be two categories, and $F, F' : \calC \to \calD$ be two functors. A natural transformation $\alpha$ from $F$ to $F'$, denoted as $\alpha : F \Rightarrow F'$ is a family of arrows $\SetDef{ \alpha_c : F(c) \to F'(c) }{ c\in ob(\calC) }$ such that the following commutations hold
	\begin{equation} \label{eqn:def:nat_transform}
		\forall \begin{tikzcd} c \arrow[d, "f"] \\ c' \end{tikzcd} , \quad 
		\begin{tikzcd}
			F(c) \arrow[r, "\alpha_{c}"] \arrow[d, "F(f)"'] & F'(c) \arrow[d, "F'(f)"] \\
			F(c') \arrow[r, "\alpha_{c'}"'] & F'(c')                        
		\end{tikzcd}
	\end{equation}
	Just like arrows / morphisms connect objects, natural transformations connect functors. In fact the collection $\Functor{\calC}{\calD}$ of functors from $\calC$ to $\calD$ is itself a category, in which the morphisms are natural transformations. This realization that the collection of functors itself is a category bears a lot of significance in the study of homology, simplicial complexes and sheaf theory in general. We will find a special significance in the study of dynamical systems.
	
	The collection of dynamical systems evolving according to time interpretation $\calT$ and within the context $\Context$, forms the functor category 
	\[\DS := \Functor{\Time}{\Context}.\]
	Each object in this category is a functor $\tilde{\Phi} : \Time \to \Context$. Every such functor automatically involves a selection of an object $\Omega$ from $\Context$ as the domain. A morphism in $\Functor{\Time}{\Context}$ is a natural transformation $h$. Due to the special structure of $\Time$, the commutation in \eqref{eqn:def:nat_transform} becomes
	\[\begin{tikzcd}
		\Omega \arrow[rr, "\Phi^t"] \arrow[d, "h"'] && \Omega \arrow[d, "h"] \\
		\Omega' \arrow[rr, "\Phi'^t"'] && \Omega'
	\end{tikzcd}, \quad \forall t\in \Time. \]
	The diagram shown here is called a \emph{semi-conjugacy} in classical dynamical systems theory. It is a generalization of the notion of change of variables. If $h$ is injective, it leads to an embedding of the dynamics of $\Phi^t$ within $\Phi'^t$. If $h$ is surjective, it means that the dynamics of $\Phi'^t$ is \emph{factored} within $\Phi'^t$.
	
	\begin{figure}[!t]
		\centering
		\includegraphics[width=0.26\textwidth]{\figs N_semigroup.pdf}
		\begin{tikzcd} {} \arrow[rrr, mapsto, "\Phi"] &&& {} \end{tikzcd}
		\includegraphics[width=0.26\textwidth]{\figs 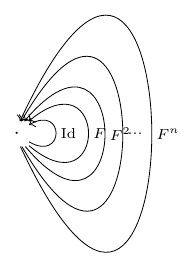}
		\caption{Dynamics as a functor. Section \ref{sec:cat} presents a dynamical system as a functor between a semigroup category $\Time$ representing time, and a category $\calC$ representing the context. Shown here is such a functor $\tilde{\Phi}$ transforming the additive semigroup $\SqBrack{\num_0, +}$ of non-negative integers into self-maps on a space $\Omega$. Since $\SqBrack{\num_0, +}$ is generated by the unit element $1$, the image $F$ of this element generates all iterations of the dynamics. } 
		\label{fig:dyn_functor}
	\end{figure}
	
	This completes the very basic step of our categorical language - interpreting dynamics as functors and semi-conjugacies as natural transformations. We next present the need of the finer structure of enriched categories and functors.
	
	%-_-_-_-_-_-_-_-_-_-_-_-_-_-_-_-_-_-_-_-_-_-_-_-_-_-_-_-_-_-_-_-_-_-_-_-_-_-_-_-_-_-_-_-_-_-_-_-_-_-_-_-_-_-_-_-_-_-_-_-_-_-_-_-_-_-_-_-_-_-_-_-_-_-_-_-
	\section{Time as an enriched category} \label{sec:time}	
	
	In the first description of dynamics, time is merely assumed to be a semigroup. There is no additional assumption imposed on the collection of time-instants, which could make them interpretable in context $\Context$. In most studies of dynamical systems, the dependence of the transformation of time is assumed to be more than just the result of the homomorphism $\tilde{\Phi}$. For example in ODE-theory, the flow $\Phi$ is differentiable w.r.t. time, and in the study of operator semigroups, and in measurable dynamics, the dependence on time is measurable too. These properties cannot be guaranteed by Assumption \ref{A:dyn} alone. To enable the time-semigroup itself to be present as an object of the context category $\Context$, we need to rely on advanced categorical notions such as monoidal categories and enriched categories. 
	
	A \textit{monoidal category} is a category equipped with a product like structure $\otimes$, called \emph{tensor product}. It allows objects multiple objects $a, b, c, \ldots$ to be combined into a new object $a\otimes b \otimes c \otimes \ldots$. A tensor product is a bifunctor which generalizes the usual notion of categorical product. This operation is associative up to isomorphism, and has a identity element called the \emph{unit object}, and is denoted by $1_{\Context}$. A monoidal category is said to be \emph{symmetric} if the tensor product is symmetric up to a natural isomorphism. i.e., there exists a natural isomorphisms $\swap_{a,b} : a\otimes b \to b\otimes a$. Our next assumption will be :
	
	\begin{Assumption} \label{A:monoidal}
		The context category $\Context$ is a symmetric monoidal category, and $\Time$ is enriched over $\Context$.
	\end{Assumption}
	
	Many classical settings of dynamical systems theory, such as the category of topological spaces and continuous maps, are closed monoidal categories. Further, the theory of Markov processes can also be developed within a symmetric monoidal category.
	\begin{definition} [Category of Markov processes]
		The category $\StochCat$ contains as objects measurable spaces $(\Omega, \Sigma)$, and Markov transitions as morphisms. Thus a morphism from $(\Omega, \Sigma)$ to $(\Omega', \Sigma')$ is a map $k : \Omega \times \Sigma' \to [0,1]$ such that
		\begin{enumerate} [(i)]
			\item for each $\omega\in \Omega$, $k(\omega, \cdot) : \Sigma' \to [0,1]$ is a probability measure;
			\item for each $S\in \Sigma'$, $k(\cdot, S) : \Omega \to [0,1]$ is a measurable function.
		\end{enumerate}
		The composite of two kernels $k, k'$ is the kernel $k''$ depicted below
		\[\begin{tikzcd}
			& (\Omega', \Sigma') \arrow[dr, "k'"] \\
			(\Omega, \Sigma) \arrow[dashed, rr, "k''"] \arrow[ur, "k"] && (\Omega'', \Sigma'')
		\end{tikzcd}\]
		defined as 
		\[ k''(\omega, S) = \int_{\omega' \in \Omega'} k'(\omega', S) d k(\omega, d\omega') , \quad \forall \omega\in \Omega, \, S \in \Sigma'' .\]
	\end{definition}
	
	Recall that a category $\calX$ is said to be \emph{enriched} over the monoidal category $\Context$ if the following are satisfied :
	\begin{enumerate}[(i)]
		\item for every $a,b\in ob(\Context)$, there is an object $\calX(a,b) \in ob(\Context)$;
		\item for every $f\in \Hom_{\calX}( a; b)$, there is a morphism $ f_{\Context} \in \Hom_{\Context} \left( 1_{\Context}; \calX(a,b) \right)$. If $a=b$ and $f = \Id_a$, then the corresponding morphism is denoted as $\start_a : 1_{\Context} \to \calX(a,a) $;
		\item for every $a,b,c \in ob(C)$, there is a map $\circ_{a,b,c} \in \Hom_{\Context}\left( \calX(b,c) \otimes_{\Context} \calX(a,b); \calX(c,a) \right)$ such that the following three commutations hold :
		\[\begin{tikzcd} [column sep = huge]
			\left( \calX(c,d) \otimes_{\Context} \calX(b,c) \right) \otimes_{\Context} \calX(a,b) \arrow{d}[swap]{\text{assoc}} \arrow{r}{\circ_{b,c,d} \otimes_{\Context} \Id_{\calX(a,b)} } & \calX(b,d) \otimes_{\Context} \calX(a, b) \arrow{r}{\circ_{a,b,d}} & \calX(a,d) \\
			\calX(c,d) \otimes_{\Context} \left( \calX(b,c) \otimes_{\Context} \calX(a,b) \right) \arrow{r}[swap]{ \Id_{\calX(c,d)} \otimes_{\Context} \circ_{a,b,c} } & \calX(c,d) \otimes_{\Context} \calX(a,c) \arrow{ur}[swap]{\circ_{a,c,d}}
		\end{tikzcd}\]
		\[\begin{tikzcd} [column sep = huge]
			1_{\Context}\otimes_{\Context} \calX(a,b) \arrow[bend right = 10]{rr}[swap]{ \text{left identity} } \arrow{r}{\Id_b \otimes_{\Context} \Id_{\calX(a,b)}} & \calX(b,b) \otimes_{\Context} \calX(a,b) \arrow{r}{\circ_{a,b,b}} & \calX(a,b)
		\end{tikzcd}\]
		\[\begin{tikzcd} [column sep = huge]
			\calX(a,b) \otimes_{\Context} 1_{\Context} \arrow[bend right = 10]{rr}[swap]{ \text{right identity} } \arrow{r}{\Id_{\calX(a,b)} \otimes_{\Context} \Id_a } & \calX(a,b) \otimes_{\Context} \calX(a,a) \arrow{r}{\circ_{a,a,b}} & \calX(a,b)
		\end{tikzcd}\]
	\end{enumerate}
	
	In our case, the 1-object category $\Time$ plays the role of the general category $\calX$. This implies the following :  
	\begin{enumerate}[(i)]
		\item the hom-set 
		\begin{equation} \label{eqn:def:TimeB}
			\TimeB:= \Hom_{\Time}(*,*)
		\end{equation}
		is an object in the context category $\Context$.
		\item For every morphism $t\in \Hom_{\Time}(*,*)$, there is a $\Context$-morphism $f_t \in \Hom_{\Context} \left( 1_{\Context}; \TimeB \right)$. When $t = \Id_{*}$, we denote $f_t$ by $\start_{\star}$.
		\item The object $\TimeB$ is a monoid, i.e., there is a $\Context$ morphism
		\[\begin{tikzcd} \TimeB \otimes \TimeB \arrow[rr, "\add"] && \TimeB \end{tikzcd}\]
		such that the following commutations hold :
		\begin{equation} \label{eqn:TimeB:id}
			\begin{tikzcd} 
				\TimeB \otimes \TimeB \otimes \TimeB \arrow{dr}[swap]{ \TimeB \otimes \add } \arrow{rr}{ \add \otimes \TimeB } && \TimeB \otimes \TimeB \arrow{d}{\add} \\
				& \TimeB \otimes \TimeB \arrow{r}[swap]{\add} & \TimeB
			\end{tikzcd} , \quad \begin{tikzcd} 
				1_{\Context}\otimes \TimeB \arrow[bend right = 10]{drr}[swap]{ \begin{array}{c}\text{left} \\ \text{identity} \end{array} } \arrow{rr}{\start_{\star}\otimes \TimeB} && \TimeB \otimes \TimeB \arrow{d}{\add} && \TimeB \otimes 1_{\Context} \arrow{ll}[swap]{ \TimeB \otimes \start_{\star}} \arrow[bend left = 10]{dll}{ \begin{array}{c}\text{right} \\ \text{identity} \end{array} } \\
				&& \TimeB
			\end{tikzcd}
		\end{equation}
	\end{enumerate}
	
	The first diagram establishes the associativity of the composition rule. The second two diagrams uphold some basic rules of combinations of $\start_{\star}$ and $\add$. Together, these commuting diagrams preserve the algebraic structure of $\Time$ within $\calC$. The last two commutation loops is reminiscent of the idea of a \emph{dynamical system over a monoid action} \cite{BehrischEtAl2017dyn}. The ideas present in that categorification  of dynamical systems relies on a complicated algebraic structure of time. Our construct of time is significantly simpler, and the algebraic properties of the action of time is encoded within the enriched structure of the functor $\tilde{\Phi}$.
	
	Note the subtle difference in the semigroup property of $\Time$, the interpretation of time-instants as provided by \eqref{eqn:def:TimeB}, and how Assumption \ref{A:monoidal} enables this loose collection $\Endo\paran{ \star_{\calT} }$ of endomorphisms to be realized as an object in $\Context$. This means that $\TimeB$ may be connected to other $\Context$-objects with morphisms of the same type as $\Context$. The two most important among such morphisms are $\start_{\star}$ and $\add$. The play the role of an initial time-instant, and addition. The notion of an initial time and addition was encoded within the compositional structure of $\Time$. In $\TimeB$ their algebraic properties have are encoded within the identities \eqref{eqn:TimeB:id} above.
	
	%-_-_-_-_-_-_-_-_-_-_-_-_-_-_-_-_-_-_-_-_-_-_-_-_-_-_-_-_-_-_-_-_-_-_-_-_-_-_-_-_-_-_-_-_-_-_-_-_-_-_-_-_-_-_-_-_-_-_-_-_-_-_-_-_-_-_-_-_-_-_-_-_-_-_-_-
	\paragraph{Closed monoidal structure} We are accustomed to viewing functional spaces such as endomorphisms and path-spaces as topological spaces of their own. The collection of maps between two sets is itself a set. Our next assumption generalizes this property for the context $\Context$ :
	
	\begin{Assumption} \label{A:closed}
		The context category $\Context$ is a closed symmetric monoidal category. 
	\end{Assumption}
	
	While Assumption \ref{A:monoidal} already assumed $\Context$ to be monoidal, Assumption \ref{A:closed} makes the additional assumption of $\Context$ being closed too. Recall that being \emph{closed, monoidal} means that for every two objects $y, z\in \calC$ there is a special object $\IntHom{y}{z}$ called the \emph{exponential} object or \emph{internal-hom} object. It allows the collection of morphisms from $y$ to $z$ being represented as an object of $\Context$ itself. Its defining property is the following natural bijection between hom-sets
	\begin{equation} \label{eqn:closed_mnd:1}
		\Hom_{\calC} \paran{ y\otimes x; z } \xrightarrow{\cong} \Hom_{\calC} \paran{ x; \IntHom{y}{z} }.
	\end{equation}
	In categorical terms this means that the internal-hom is a right adjoint to the monoidal product. Taking $x=1_{\calC}$ in \eqref{eqn:closed_mnd:1} implies
	\begin{equation} \label{eqn:closed_mnd:2}
		\Hom_{\calC} \paran{ y; z } \xrightarrow{\cong} \Hom_{\calC} \paran{ 1_{\calC}; \IntHom{y}{z} }.
	\end{equation}
	Equation \eqref{eqn:closed_mnd:2} explains how the object $\IntHom{y}{z}$ is an alternate means of describing the Hom-set $\Hom(y;z)$. The elements of the object $\IntHom{y}{z}$ are in bijective correspondence with the collection of morphisms from $y$ to $z$. Now suppose that  we take $x=\IntHom{y}{z}$ in \eqref{eqn:closed_mnd:1}. Then we get a bijection
	\begin{equation} \label{eqn:closed_mnd:3}
		\Hom_{\calC} \paran{ y\otimes \IntHom{y}{z}; z } \xrightarrow{\cong} \Hom_{\calC} \paran{ \IntHom{y}{z}; \IntHom{y}{z} }, \quad \eval_{y,z} \mapsto \Id_{\IntHom{y}{z}} .
	\end{equation}
	As indicated, there is a morphism $\eval_{y,z} : y\otimes \IntHom{y}{z} \to z$ which corresponds to the identity morphism of $\IntHom{y}{z}$. This is called the \emph{evaluation} morphism of $y,z$. Intuitively, the evaluation morphism $\eval_{y,z}$ assigns to every \emph{internal function} $\phi \in \IntHom{y}{z}$, and evaluation point in $y$, a point in $z$. Evaluation morphisms are a categorical realization of the action of usual set-maps, but in a general categorical context. 
	
	We present two useful examples of closed monoidal categories.
	
	\begin{definition} [Contractive maps 1] \label{def:MetriCat}
		The category $\MetricCat$ of metric spaces and contractive maps are closed with respect to their Cartesian products. This is because the collection of contractive maps between two metric spaces $X, Y$ is itself a metric space with the supremum norm.
	\end{definition}
	
	\begin{definition} [Contractive maps 2] \label{def:CompactMetriCat}
		The collection of compact metric spaces form a sub-category of $\MetricCat$. It also preserved the closure property \cite[see]{Leinster2022eventual, BuragoEtAl2001metric}, and hence is another instance of a closed monoidal category.
	\end{definition}
	
	Next we take note of two related notions of adjunction, which will be crucial in our discussion.
	
	\begin{definition} [Two-types of adjoints]
		Let $\Phi : \TimeB \otimes \Omega \to \Omega$ be a morphism in a closed symmetric monoidal category. Since $\Context$ is symmetric and closed, we can take two kinds of adjoints of $\tilde{\Phi}$ :
		\begin{equation}
			\begin{aligned}
				\Phi^\sharp&: \TimeB \to \IntHom{\Omega}{\Omega} \\
				\Phi^\flat&: \Omega \to \IntHom{\TimeB}{\Omega}.
			\end{aligned}
		\end{equation}
		We call the former the \emph{sharp adjoint} and the latter the \emph{flat adjoint}.
	\end{definition}
	The sharp adjoint corresponds to the usual identification of a flow $\Phi(t,x)$ with a family of endomorphisms $\Phi^t$. It raises the time variable. Similarly, the flat adjoint corresponds to the identification  with a family of curves $\Phi_x$ starting from the initial point $x$. 
	
	\paragraph{Enriched objects} Assumption \ref{A:closed} enables many essential constructs of dynamical systems theory to be realized as objects of $\Context$. Pick an element $\Omega$ from $\Context$. An important object is
	\begin{equation} \label{eqn:def:Endo}
		\Endo (\Omega) := \IntHom{\Omega}{\Omega} ,
	\end{equation}
	the {endomorphism} object associated to domain $\Omega$. It allows the collection of self maps of $\Omega$ to be interpreted as an object in $\Context$. Another important consequence is the creation of an object 
	\begin{equation} \label{eqn:def:Path:1}
		\PathF(\Omega) := \IntHom{\TimeB}{\Omega} ,
	\end{equation}
	which we interpret as the \emph{path-space} of $\Time$-indexed paths in domain $\Omega$. It is the collection of all possible \emph{orbit}-s on the domain $\Omega$, not limited to those generated by the dynamics of $F$. Due to Assumption \ref{A:monoidal}, this collection is also of the same type / category as objects in $\Context$. It is alternatively known as the space of  \emph{sample-paths} and is a fundamental object of study in the theory of stochastic processes \cite[e.g.]{Doob1953book}.
	
	The time object $\TimeB$ and endomorphism object $\Endo(\Omega)$ together enable two new definitions of a dynamical system : 
	
	\begin{definition}[Interpretation 2. Parametric endomorphisms] \label{desc:dyn:2}
		A dynamical system is a parameterized family of maps $\Phi^\sharp : \TimeB \to \Endo(\Omega)$.
	\end{definition}
	
	Thus this interpretation of a dynamical system is as a morphism 
	\begin{equation} \label{eqn:def:hatPhi}
		\Phi^\sharp:\TimeB \to \IntHom{\Omega}{\Omega}
	\end{equation}
	Now note that by Assumption \ref{A:monoidal}, $\Phi^\sharp$ has an adjoint morphism 
	\begin{equation} \label{eqn:def:tildePhi}
		\Phi:\TimeB \otimes \Omega \to \Omega .
	\end{equation}

	\begin{definition} [Interpretation 3. Pre-flow] \label{desc:dyn:3}
		A dynamical system is a pre-flow on the domain $\Omega$, meaning a morphism / map $\Phi : \TimeB \times \Omega \to \Omega$.
	\end{definition}
	
	Thus we have two different interpretations of dynamics- as a  parameterized family $\Phi^\sharp$ of self-maps \eqref{eqn:def:hatPhi} -- and as as a pre-flow \eqref{eqn:def:tildePhi}. 
	We shall impose some algebraic conditions on a pre-flow to obtain a \emph{flow} or a \emph{left-action} of a semi-group, in Definition \ref{desc:dyn:5}.
	At this point the three descriptions are different and not interchangeable, in a general category $\Context$. For example, the flow map $\Phi^t$ generated by a vector field leads to an action of $\real$ on the group of diffeomorphisms, and the dependence on $t$ is smooth. However, a time rescaling such as $t \mapsto \Phi^{t^2}$ is still a dynamical system according to Interpretations 2 and 3, but not according to Interpretation 1. We shall need more assumptions on $\Context$ and $\tilde{\Phi}$ to enable these descriptions to be  interchangeable. We end this section by presenting three important dynamical systems that can be constructed out of a pre-flow $\Phi$.
	
	\paragraph{Secondary dynamics} We assume Assumptions \ref{A:monoidal} and \ref{A:closed} hold in the following constructions. 
	
	\begin{enumerate}
		\item \textbf{Shift map.} One of the most important instances of a flow is provided by a natural action of time on the elements of the path space. Recall from \eqref{eqn:def:Path:1} that $\PathF(\Omega)$ is just the internal hom object $\IntHom{\TimeB}{\Omega}$. Consider the following composite morphism shown as a dashed arrow below on the left :
		\begin{equation} \label{eqn:def:shift}
			\begin{tikzcd}
				\TimeB \otimes \TimeB \otimes \IntHom{\TimeB}{\Omega} \arrow[r, dashed] \arrow{dr}[swap]{ \add \otimes \IntHom{\TimeB}{\Omega} } & \Omega \\
				& \TimeB \otimes \IntHom{\TimeB}{\Omega} \arrow{u}[swap]{ \eval_{T, \Omega} }
			\end{tikzcd}
			\begin{tikzcd} [scale cd = 0.8] {} \arrow[rr, harpoon, shift left = 1, "\text{right adj.}"] && {} \arrow[ll, harpoon, shift left = 1, "\text{left adj.}"] \end{tikzcd}
			\begin{tikzcd} 
				\TimeB \otimes \IntHom{\TimeB}{\Omega} \arrow{d}[dashed]{\shift_{\Omega}} \\ \IntHom{\TimeB}{\Omega}
			\end{tikzcd}
		\end{equation}
		By \eqref{eqn:closed_mnd:1} this arrow corresponds naturally to the dashed arrow shown above on the right. Thus for each object $\Omega$, the path space $\PathF(\Omega)$ has a natural dynamics in the sense of Description \ref{desc:dyn:3} : 
		\[\shift: \TimeB \otimes \PathF(\Omega) \to \PathF(\Omega).\]
		The shift map $\shift$ acts on the element of the path-space via $\add$. The morphism $\add$ contains the semigroup composition rule of $\Time$. Equation \eqref{eqn:def:shift} reveals how this is transformed into a compositional rule on the paths with $\Time$-structure. Theorem \ref{thm:1} presents the functorial nature of this morphism.
		\item \textbf{Transfer operator} Now take any object $X\in \calC$. Equation \eqref{eqn:def:tildePhi} leads to the following composite morphism shown on the left :
		\begin{equation} \label{eqn:def:transfer}
			\begin{tikzcd} [scale cd = 0.7]
				X \otimes \TimeB \otimes \IntHom{X}{\Omega} \arrow[dashed, d] \arrow{rrr}{\swap \otimes \IntHom{X}{\Omega}} &&& \TimeB \otimes X \otimes \IntHom{X}{\Omega} \arrow{d}{ \TimeB \otimes \eval_{X, \Omega} } \\
				\Omega &&& \TimeB \otimes \Omega \arrow[lll, "\Phi"']
			\end{tikzcd}
			\begin{tikzcd} [scale cd = 0.8] {} \arrow[rr, harpoon, shift left = 1, "\text{right adj.}"] && {} \arrow[ll, harpoon, shift left = 1, "\text{left adj.}"] \end{tikzcd}
			\begin{tikzcd} [scale cd = 0.7] \TimeB \otimes \IntHom{X}{\Omega} \arrow[d, "\Transfer{\Phi}"] \\ \IntHom{X}{\Omega} \end{tikzcd}
		\end{equation}
		The other two morphisms shown above are obtained by successive adjoint operations. We call $\Transfer{\Phi}$ the \emph{transfer operator} corresponding to the dynamics $\tilde{\Phi}$. The object $\IntHom{X}{\Omega}$ can be interpreted as the collection of $X$-patterns in $\Omega$. The action of $\tilde{\Phi}$ on $\Omega$ induces an action on such patterns. Theorem \ref{thm:2} presents the functorial nature of this morphism. 
		\item \textbf{Koopman operator} Similar to the transfer operator, one has a Koopman operator, which act on patterns from $\Omega$ into an object $X$.
		\[\begin{tikzcd} [scale cd = 0.8]
			\TimeB \otimes \Omega \otimes \IntHom{\Omega}{X} \arrow{d}[swap]{ \Phi \otimes \IntHom{\Omega}{X} } && \Omega \otimes \TimeB \otimes \IntHom{\Omega}{X} \arrow{ll}[swap]{\swap \otimes \IntHom{\Omega}{X}} \arrow[d, dashed] \\ \Omega \otimes \IntHom{\Omega}{X} \arrow{rr}[swap]{ \eval_{X, \Omega} } && X
		\end{tikzcd}
		\begin{tikzcd} [scale cd = 0.8] {} \arrow[rr, harpoon, shift left = 1, "\text{right adj.}"] && {} \arrow[ll, harpoon, shift left = 1, "\text{left adj.}"] \end{tikzcd}
		\begin{tikzcd} [scale cd = 0.8]
			\TimeB \otimes \IntHom{\Omega}{X} \arrow{d}{U_{\Phi} } \\ \IntHom{\Omega}{X}
		\end{tikzcd}
		\]
		In other words $U_{\Phi}$ is a pre-flow in the sense of Definition \ref{desc:dyn:3}.
	\end{enumerate}
	
	So far we have seen how an additional structural assumption on $\Context$ leads to an interpretation of the collection of endomorphisms and paths as objects of $\Context$ itself. This also enables the creation of the $\shift$ morphism, which is a dynamical system according to the interpretation in Definition \ref{desc:dyn:3} of a dynamical system. We have also seen the creation of the transfer and Koopman operators, which are dynamical systems not according to Definition \ref{desc:dyn:1} but according to Definitions \ref{desc:dyn:2} and \ref{desc:dyn:3}. To make the two new interpretations, along with $\shift$, Koopman and transfer morphisms conform with Definition \ref{desc:dyn:1}, we need yet another structural assumption, to be described next.
	
	%-_-_-_-_-_-_-_-_-_-_-_-_-_-_-_-_-_-_-_-_-_-_-_-_-_-_-_-_-_-_-_-_-_-_-_-_-_-_-_-_-_-_-_-_-_-_-_-_-_-_-_-_-_-_-_-_-_-_-_-_-_-_-_-_-_-_-_-_-_-_-_-_-_-_-_-
	\section{Enriched functors} \label{sec:enrich_fnctr} 
	
	\begin{figure} [!t]
		\center
		\begin{tikzpicture}[scale=0.6, transform shape]
			\input{\figs dyn_interpret.tex}
		\end{tikzpicture}
		\caption{Many interpretations of dynamical systems. Dynamics are interpreted in different ways in different areas of mathematics. We start by interpreting dynamics as the action of a semigroup. The diagram depicts how this interpretation could lead to three other interpretations $\Phi^\sharp$, $\Phi$ and $\shiftF{\Phi}$ respectively. This is done by combining with different assumptions on the context category $\Context$, time semigroup $\Time$, and functor $\tilde{\Phi}$. These assumptions shown in the center of the diagram are categorical in nature. They lead to the basic concepts shown in dark boxes. Any discussion of a dynamical system requires a consideration of three associated objects - the abstract algebraic notion of time; the collection of $\Time$-indexed paths in $\Context$; and the collection of self-maps on a domain $\Omega$ with $\Context$. If $\Context$ is closed, all of these can now be interpreted as objects of $\Context$ itself. The arrows in the diagram represents logical implications. Each box represents a concept or interpretation. It can be constructed by assuming the contents of its preceding boxes. Thus the interpretation $\Phi^\sharp$ as a parameterized family requires the concept of an time-object, self-map object, and the assumption of $\tilde{\Phi}$ being an enriched functor. The interpretation $\Phi$ as a pre-flow then follows via an adjoint relation guaranteed by the closed monoidal structure of $\calC$. The interpretation $\shiftF{\Phi}$ as a sub-shift is a direct consequence of the pre-flow-interpretation, time-object and self-maps. The top-left corner of the diagram shows that in fact, if the context $\Context$ can be placed in $\SetCat$, then one obtains a dynamics into $\SetCat$ by composition. In the classical analytic approach to dynamics, the functor $\tilde{\Phi}$ is replaced by a continuous family of self-maps. As a result, it is easier to switch between the 4 interpretations. In our categorical setting, we need the crucial Assumptions on enrichment to make the connect between Interpretation 1 and 2. This identification of the structural components allows a synthetic approach to dynamical systems theory.} 
		\label{fig:dyn_interpret}
	\end{figure} 
	
	Assumptions \ref{A:dyn} and \ref{A:monoidal} enable a description of the dynamics  but that only captures the semigroup property. Depending on the context such as those listed in Table \ref{tab:TC}, one often needs to establish the dependence of the dynamics law on time as continuous, linear, measurable or differentiable. Thus there is a need to extend the mathematical properties encoded within context $\Context$ to the correspondence $t \mapsto \Phi(t)$ as well. A general functor $\tilde{\Phi} : \Time \to \Context$ from Assumption \ref{A:dyn} does not respect the $\Context$-enriched structure of $\Time$. As a result, $\tilde{\Phi}$ does not automatically lead to a morphism  involving the time object $\TimeB$ within the category $\Context$. What we need is the concept of an \emph{enriched functor}.
	
	An \emph{enriched functor} $F : \calX \to \calY$ between two categories enriched over a monoidal category $\calM$ is an ordinary functor such that between any two pairs of objects $x, x'\in \calX$, there is an induced $\calM$ morphism 
	\begin{equation} \label{eqn:def:enrich_func:1}
		F_{x, x'} : \calX \paran{ x; x' } \to \calY \paran{ Fx; Fx' } ,
	\end{equation}
	which respects unital composition  : 
	\begin{equation} \label{eqn:def:enrich_func:2}
		\begin{tikzcd}
			1_{\calC} \arrow{rr}{ \iota_{Fx} } \arrow{drr}[swap]{ \iota_{x} } && \calY \Comma{Fx}{Fx} \\
			&& \calX \Comma{x}{x} \arrow{u}[swap]{ F_{x,x} }
		\end{tikzcd}
	\end{equation}
	and compositionality :
	\begin{equation} \label{eqn:def:enrich_func:3}
		\begin{tikzcd}
			\calX \Comma{x'}{x''} \otimes \calX \Comma{x}{x'} \arrow{rr}{ \circ_{x, x', x''} } \arrow{d}[swap]{ F_{x', x''} \otimes F_{x, x'} } && \calX \Comma{x}{x''} \arrow{d}{ F_{x, x''} } \\
			\calY \Comma{Fx'}{Fx''} \otimes \calX \Comma{Fx}{Fx'} \arrow{rr}[swap]{ \circ_{Fx, Fx', Fx''} } && \calY \Comma{Fx}{Fx''} \\
		\end{tikzcd}
	\end{equation}
	When $\calY = \calM = \Context$ and $\calX = \Time$, then $\calX$ has only one object. Let $\Omega$ be its image under $F$. Then \eqref{eqn:def:enrich_func:1} comprises of only one morphism
	\begin{equation} \label{eqn:def:enrich_func:Time:1}
		F^\sharp : \TimeB \to \Endo(\Omega) ,
	\end{equation}
	and the conditions \eqref{eqn:def:enrich_func:2}  and \eqref{eqn:def:enrich_func:3} become :
	\begin{equation}\label{eqn:def:enrich_func:Time:2}
		\begin{tikzcd}
			1_{\calC} \arrow{rr}{ \iota_{\Omega} } \arrow{drr}[swap]{ \iota_{T} } &&  \Endo(\Omega) \\
			&& \TimeB \arrow{u}[swap]{ F^\sharp }
		\end{tikzcd} ,\quad 
		\begin{tikzcd}
			\TimeB \otimes \TimeB \arrow{rr}{ \circ } \arrow{d}[swap]{ F^\sharp \otimes F^\sharp } && \TimeB \arrow{d}{ F^\sharp } \\
			\Endo(\Omega) \otimes \Endo(\Omega) \arrow{rr}[swap]{ \circ } && \Endo(\Omega) \\
		\end{tikzcd}
	\end{equation}
	Let $F,G : \calX \to \calY$ be two enriched functors. Then an \emph{enriched natural transformation} $\omega : F \Rightarrow G$ is a family of morphisms
	\begin{equation} \label{eqn:def:enrich_nat:1}
		\omega_x : 1_{\calC} \to \calY \Comma{Fx}{Gx} , \quad \forall x \in \calX,
	\end{equation}
	such that the following commutations hold :
	\begin{equation} \label{eqn:def:enrich_nat:2}
		\forall x,x'\in \calX \;:\;
		\begin{tikzcd}
			\calX \Comma{x}{x} \otimes 1_{\calC} \arrow{rr}{ G_{x, x'} \otimes \omega_x } && \calY \Comma{Gx}{Gx'} \otimes \calY \Comma{Fx}{Gx} \arrow{d}{ \circ^{\calC}_{ Fx, Gx, Gx' } } \\
			\calX \Comma{x}{x} \arrow{u}{\cong} \arrow{d}[swap]{\cong} && \calY \Comma{Fx}{Gx'} \\
			1_{\calC} \otimes \calX \Comma{x}{x} \arrow{rr}[swap]{ \omega_{x'} \otimes F_{x, x'} } && \calY \Comma{Fx'}{Gx'} \otimes \calY \Comma{Fx}{Fx'} \arrow{u}[swap]{ \circ^{\calC}_{ Fx, Fx', Gx' } }
		\end{tikzcd}
	\end{equation}

	We would like to assume that $\tilde{\Phi}$ is an enriched functor. We would require both its domain $\Time$ and co-domain $\Context$ to be categories enriched over the same monoidal category, which in this case is also $\Context$. The domain $\Time$ is enriched by Assumption \ref{A:monoidal}. The codomain $\Context$ is automatically enriched, as observed below :
	
	\begin{lemma} \label{lem:dj9da3}
		\cite[Sec 1.6]{Kelly1982enriched} A closed monoidal category is enriched over itself. The internal hom operation plays the role of the enrichment.
	\end{lemma}
	
	This observation enables the additional assumption :
	
	\begin{Assumption} \label{A:enrich_fnctr}
		The functor $\tilde{\Phi}$ from Assumption \ref{A:dyn} is an enriched functor.
	\end{Assumption}
	
	Assumption \ref{A:enrich_fnctr} enables the semigroup action $\tilde{\Phi}$ to be interpreted as a parameterized family of self-maps. Note that every category is enriched over $\SetCat$, and so is every ordinary functor. Thus Assumptions \ref{A:monoidal}, \ref{A:closed} and \ref{A:enrich_fnctr} is automatically satisfied if $\Context = \SetCat$. Given an enriched functor $\tilde{\Psi} : \Time \to \Context$, equations \eqref{eqn:def:enrich_func:1}, \eqref{eqn:def:enrich_func:2} and \eqref{eqn:def:enrich_func:3} imply the existence of an object $\Omega$ of $\Context$ and a morphism $\Psi^{\sharp} : \TimeB \to \Endo(\Omega)$ such that the following commutations hold : 
	\begin{equation} \label{eqn:enrch_fnctr_time:1}
		\begin{tikzcd}
			1_{\Context} \arrow{r}{ \start_{\star} } \arrow{dr}[swap]{ \start_{\Omega} } & \TimeB \arrow{d}{ \Psi^{\sharp} } \\
			& \Endo (\Omega)
		\end{tikzcd} , \quad 
		\begin{tikzcd}
			\TimeB \otimes \TimeB \arrow{d}[swap]{ \Psi^{\sharp} \otimes \Psi^{\sharp} } \arrow{rr}{\add} && \TimeB \arrow{d}{ \Psi^{\sharp} } \\
			\Endo(\Omega) \otimes \Endo(\Omega) \arrow{rr}{ \circ_{\Omega, \Omega, \Omega} } && \Endo(\Omega)
		\end{tikzcd}
	\end{equation}
	Here, the composition morphism $\circ$ is defined by the adjoint of
	\[\begin{tikzcd}
		\IntHom{\Omega}{\Omega} \otimes \IntHom{\Omega}{\Omega} \otimes \Omega \arrow{rrr}{ \IntHom{\Omega}{\Omega} \otimes \eval_\Omega} &&& \IntHom{\Omega}{\Omega} \otimes \Omega \arrow{rr}{\eval_\Omega} && \Omega.
	\end{tikzcd}\]
	We call $\Psi^\sharp$ the \textit{parametric form} of the enriched functor $\tilde{\Psi}$.
	
	\begin{lemma} \label{lem:enrch:1} 
		Let Assumptions \ref{A:dyn},  \ref{A:monoidal}, \ref{A:closed} and \ref{A:enrich_fnctr} hold.
		\begin{enumerate} [(i)]
			\item Given any semigroup $\Time$ enriched over a closed monoidal category $\Context$, an enriched functor from $\tilde{\Psi} : \Time \to \Context$ corresponds uniquely with a morphism $\Psi^\sharp : \TimeB \to \Endo(\Omega)$ that satisfies \eqref{eqn:enrch_fnctr_time:1}. 
			\item A morphism $\Psi : \TimeB \otimes \Omega \to \Omega$ is the the left adjoint of a parametric form of an enriched functor iff
			\begin{equation} \label{eqn:enrch_fnctr_time:2}
				\begin{tikzcd}
					1_{\Context} \otimes \Omega \arrow{dr}[swap]{ \start_{\star} \otimes \Omega } & \Omega \arrow{l}[swap]{=}  \\
					& \TimeB \otimes \Omega \arrow{u}[swap]{ \Psi }
				\end{tikzcd} , \quad \begin{tikzcd}
					\TimeB \otimes \TimeB \otimes \Omega \arrow{d}[swap]{ \add \otimes \Omega } \arrow{rr}{\TimeB \otimes {\Psi}} && \TimeB \otimes \Omega \arrow{d}{ {\Psi} } \\
					\TimeB \otimes \Omega \arrow{rr}{ {\Psi} } && \Omega
				\end{tikzcd}
			\end{equation}  
			%That is, $\Psi$ is a left $\TimeB$-action. 
		\end{enumerate}
	\end{lemma}
	
	Lemma \ref{lem:enrch:1} follows from the basic theory of enriched categories. One needs to place $\Time$ and $\Context$ in the role of $\calX$ and $\calY$ respectively, in the definition of enriched functors. Take $\tilde{\Phi}$ as $F$ in \eqref{eqn:def:nat_transform}, $x,x'$ to be the unique object of $\Time$, and $y,y'$ to be the image $\Omega$ of $\tilde{\Phi}$.  The details of the proof can be found in Section \ref{sec:proof:enrch_fnctr_time}. Any morphism $\Psi^\sharp : \TimeB \to \Endo{\Omega}$ which also satisfies \eqref{eqn:enrch_fnctr_time:1} will be called a \emph{left action}. 
	Our first main result is that the shift morphism along the path space is the flow representation of an enriched functor : 
	
	\begin{theorem} [Semigroup action of shift operator] \label{thm:1}
		Let Assumptions \ref{A:dyn}, \ref{A:monoidal} and \ref{A:closed} hold. The shift morphism from \eqref{eqn:def:shift} describes a left action of $\Time$ on the object $\PathF(\Omega)$ of $\Context$. Thus it corresponds uniquely to an enriched functor from $\Time$ to $\Context$, with domain $\PathF(\Omega)$. % and \ref{A:enrich_fnctr}
	\end{theorem}
	
	Theorem \ref{thm:1} is proved in Section \ref{sec:proof:thm:1}. Although the shift flow is constructed for a prefixed object $\Omega$ as in \eqref{eqn:def:shift}, Theorem \ref{thm:1} indicates that it preserves the structure present within $\Context$. The shift flow is also intrinsic to an object $\Omega$, and we shall explore the functorial nature of this relation later. Theorem \ref{thm:1} establishes the the shift flow $\shift$ to be an object of the following category --
	
	\begin{definition} [Enriched semigroup actions] \label{desc:dyn:4}
		The category of enriched semigroup actions $\EnrichSGActCat{\Context}$ based on the context category $\Context$ and time semigroup $\calT$ contains as objects enriched functors from $\Time$ to $\Context$, and enriched natural transformations as morphisms.
	\end{definition}
	
	In our case $\calX = \Time$ which has only one object. So in  diagram \eqref{eqn:def:enrich_nat:2} we have	
	\[\begin{split}
		&x=x' \\
		&\calY \paran{ Gx; Gx' } = \Endo(\Omega'), \\ 
		&\calY \paran{ Fx; Gx } = \calY \paran{ Fx; Gx' } = \calY \paran{ Fx'; Gx' } = \IntHom{\Omega}{\Omega'}, \\
		&\calY \paran{ Fx; Fx' } = \Endo(\Omega).
	\end{split}\]
	Consider two  objects $\tilde{\Phi}, \tilde{\Phi'}$ in $\EnrichSGActCat{\Context}$ with domains $\Omega, \Omega'$ respectively. An $\EnrichSGActCat{\Context}$-morphism between them is a morphism $h:\Omega\to \Omega'$ such that its right-adjoint $h^\sharp : 1_{\Context} \to \IntHom{\Omega}{\Omega'}$ satisfies :
	\begin{equation} \label{eqn:EnrichSGActCat:morphism}
		\begin{tikzcd}
			\TimeB \otimes 1_{\Context} \arrow{rr}{ \tilde{\Phi}'^\sharp \otimes h^\sharp } && \Endo(\Omega') \otimes \IntHom{\Omega}{\Omega'} \arrow{dr}{ \circ } \\
			\TimeB \arrow{d}{\cong}[swap]{ \lambda_{\TimeB}^{-1} } \arrow{u}{ \rho_{{\TimeB}}^{-1} }[swap]{\cong} && & \IntHom{\Omega}{\Omega'} \\
			1_{\Context} \otimes \TimeB \arrow{rr}{ h^\sharp \otimes \tilde{\Phi}^\sharp } && \IntHom{\Omega}{\Omega'} \otimes \Endo(\Omega) \arrow{ur}{ \circ }
		\end{tikzcd}
	\end{equation}
	In $\EnrichSGActCat{\Context}$, the composition of two morphisms $h: \tilde{\Phi} \to \tilde{\Phi'}$ and $h': \tilde{\Phi'} \to \tilde{\Phi''}$ is the dashed arrow created by composition below :
	\begin{equation} \label{eqn:EnrichSGActCat:compose}
		\begin{tikzcd}
			1_{\Context} \arrow[dashed]{d}[swap]{(h' \circ h)^\sharp} \arrow{r}{\lambda_{1_{\Context}}} & 1_{\Context} \otimes 1_{\Context} \arrow{d}{h'^\sharp\otimes h^\sharp}  \\
			\IntHom{\Omega}{\Omega''} & \IntHom{\Omega'}{\Omega''} \otimes \IntHom{\Omega}{\Omega'} \arrow[l, "\circ"']
		\end{tikzcd}
	\end{equation}
	Note how the composition involves a conversion to the right adjoint, utilization of the natural composition of the internal hom objects, followed by a conversion back to the left adjoint.
	There are two other categories of dynamics we can build :
	
	\begin{definition} [Parameterized endomorphisms] \label{desc:dyn:5}
		The category of parameterized endomorphisms $\ParaEndoCat{\Context}$ based on the context category $\Context$ and time semigroup $\calT$ contains as objects morphisms $\Psi^\sharp : \TimeB \to \Endo \paran{ \Omega }$ which satisfy \eqref{eqn:enrch_fnctr_time:1}. 
	\end{definition}
	
	The interpretation in Definition \ref{desc:dyn:5} forfeits the algebraic property within the dynamics law and directly encodes the time dependence as a morphism within $\Context$. Thus the type of dependence on time, namely linear, continuous, measurable or differentiable, can be inferred directly from the context $\Context$. The morphisms and compositions within this category will be described in the upcoming Theorem \ref{thm:8}. They will be defined indirectly via the following well defined category :
	
	\begin{definition} [Category of left-actions] \label{desc:dyn:6}
		The category $\FlowCat{\Context}$ of flows or left-actions based on the context category $\Context$ and time semigroup $\calT$ contains as objects morphisms $\Psi : \TimeB \otimes \Omega \to \Omega$ which satisfy \eqref{eqn:enrch_fnctr_time:2}. A morphism between two objects $\paran{\Omega, \Psi}$ and $\paran{\Omega', \Psi'}$ corresponds to a morphism $h : \Omega \to \Omega'$ such that the following commutation holds :
		\begin{equation} \label{eqn:def:dyn:6}
			\begin{tikzcd}
				\TimeB \otimes \Omega \arrow{d}[swap]{ \TimeB \otimes h } \arrow{r}{\Psi} & \Omega \arrow{d}{h } \\
				\TimeB' \otimes \Omega' \arrow{r}{\Psi'} & \Omega'
			\end{tikzcd}
		\end{equation}
	\end{definition} 
	
	Similar to the interpretation in Definition \ref{desc:dyn:5}, Definition \ref{desc:dyn:6} also interprets the dynamics law as a morphism from time-space object $\TimeB \otimes \Omega$ into $\Omega$. Unlike the former, the latter description explicitly encodes the dependence of the dynamics on both the state as well as time objects.
	
	\begin{theorem} \label{thm:8}
		Suppose Assumptions \ref{A:dyn}, \ref{A:monoidal}, \ref{A:closed} and \ref{A:enrich_fnctr} hold. Then :
		\begin{enumerate} [(i)]
			\item The collection of objects and morphisms in $\EnrichSGActCat{\Context}$ from Definition \ref{desc:dyn:4}, and in $\FlowCat{\Context}$ from Definition \ref{desc:dyn:6} are categories.
			\item These two categories are isomorphic.
			\item There is a bijective correspondence between the objects of these categories and that of the collection $\ParaEndoCat{\Context}$ from Definition \ref{desc:dyn:5}.
			\item These three collections form three isomorphic categories as shown in \eqref{eqn:thm:8}, offering differing descriptions of dynamics in a context $\Context$.
		\end{enumerate}
		\begin{enumerate}[(i), resume]
			\item The correspondence between $\tilde{\Phi}$ and $\Phi^\sharp$ from \eqref{eqn:def:hatPhi}; and between $\Phi^\sharp$ and $\Phi$ from \eqref{eqn:def:tildePhi} are functorial.
		\end{enumerate}
	\end{theorem}
	
	Theorem \ref{thm:8} is proved in Section \ref{sec:proof:thm:8}. Each description of a dynamical system in Diagram \eqref{eqn:thm:8} differs in the nature of $\tilde{\Phi}$. Theorem \ref{thm:8} states that not only is there a one-to-one correspondence between these different $\tilde{\Phi}$s, but also between the morphisms which bind them within their own categories. Theorem \ref{thm:8} thus enables us to freely chose a particular categorical presentation of a dynamical system that suits our purpose. We next refine our understanding of the shift functor. 
	
	\begin{theorem} \label{thm:ijd9}
		Let the assumptions of Theorem \ref{thm:1} hold. Then the correspondence of $\Omega$ with $\shift_{\Omega}$ is the result of a functor
		\[\shift : \Context \to \FlowCat{\Context}.\]
		The action of this functor on morphisms is given by
		\begin{equation} \label{eqn:ijd9}
			\begin{tikzcd} \Omega \arrow[d, "f"] \\ \Omega' \end{tikzcd} 
			\begin{tikzcd} {} \arrow[rr, mapsto] && {} \end{tikzcd} 
			\begin{tikzcd}
				\TimeB \otimes \IntHom{\TimeB}{\Omega} \arrow[d, "\shift_{\Omega}"'] \arrow[rr, "\TimeB \otimes \IntHom{\TimeB}{f}"] && \TimeB \otimes \IntHom{\TimeB}{\Omega'} \arrow[d, "\shift_{\Omega'}"] \\
				\IntHom{\TimeB}{\Omega} \arrow[rr, "\IntHom{\TimeB}{f}"] && \IntHom{\TimeB}{\Omega'}
			\end{tikzcd}
		\end{equation}
	\end{theorem}
	Equation \eqref{eqn:ijd9} is proved in Section \ref{sec:proof:ijd9}. Our next important result is : 
	
	\begin{theorem} [Semigroup action of Transfer operator] \label{thm:2}
		Suppose Assumptions \ref{A:dyn}, \ref{A:monoidal}, \ref{A:closed} and \ref{A:enrich_fnctr} holds.     The morphism $\Transfer{\Phi}$ from \eqref{eqn:def:transfer} satisfies \eqref{eqn:enrch_fnctr_time:1}. Thus $\Transfer{\Phi}$ is the flow-form of an enriched functor $\Time \to \Context$ whose domain is the path space $\PathF(\Omega)$ of the domain $\Omega$ of the original functor $\tilde{\Phi}$ from Assumption \ref{A:dyn}.
	\end{theorem}
	
	Theorem \ref{thm:2} is proved in Section \ref{sec:proof:thm:2}. Theorem \ref{thm:2} shows that the transfer operator is an object of the category $\ParaEndoCat{\Context}$. By Theorem \ref{thm:8}, the transfer operator can also be interpreted as a unique enriched functor from $\Time$  to $\Context$, as well as a unique left $\TimeB$-action on $\Context$.
	
	We have thus seen three different kinds of dynamical systems within a context category $\Context$ satisfying Assumptions \ref{A:monoidal}, \ref{A:closed} and \ref{A:enrich_fnctr}. The first is a general enriched functor from $\Time$. The second is the shift-flow on path spaces, and the third is a the transfer operator. The latter two are both secondary dynamical systems on path spaces. The categorification of the collection of enriched functors, parametric dynamics is utilized in the next section  to obtain deeper insights into the original dynamics $\tilde{\Phi}$.
	
	%-_-_-_-_-_-_-_-_-_-_-_-_-_-_-_-_-_-_-_-_-_-_-_-_-_-_-_-_-_-_-_-_-_-_-_-_-_-_-_-_-_-_-_-_-_-_-_-_-_-_-_-_-_-_-_-_-_-_-_-_-_-_-_-_-_-_-_-_-_-_-_-_-_-_-_-
	\section{Path-spaces as dynamics} \label{sec:path}
	
	We have so far seen multiple descriptions of a dynamical system - as a group action, a flow, and a parameterized endomorphism. Figure \ref{fig:dyn_interpret} summarizes how our ground assumptions make these descriptions interchangeable. Each of these descriptions come handy under different circumstances. In this section we shall look at a prime example.  Recall \eqref{eqn:def:transfer} in which the transfer operator morphism $\Transfer{\Phi}$ is derived as an adjoint.  Also recall the $\shift$ morphism from \eqref{eqn:def:shift}. One can obtain other adjoints of these morphisms :
	\[\begin{tikzcd} \TimeB \otimes \IntHom{\TimeB}{\Omega} \arrow{d}[dashed]{\shift} \\ \IntHom{\TimeB}{\Omega} \end{tikzcd}
	\begin{tikzcd} [scale cd = 0.7] {} \arrow[rr, leftrightarrow, "\text{adjoint}"] && {} \end{tikzcd} 
	\begin{tikzcd} \IntHom{\TimeB}{\Omega} \arrow{d}[dashed]{\shift^\flat} \\ \IntHom{ \TimeB }{ \IntHom{\TimeB}{\Omega} } \end{tikzcd} , \quad 
	\begin{tikzcd}  \TimeB \otimes \IntHom{ \TimeB }{\Omega} \arrow[d, "\tilde{ \Transfer{\Phi} }"] \\ \IntHom{\TimeB}{\Omega} \end{tikzcd}
	\begin{tikzcd} [scale cd = 0.7] {} \arrow[rr, leftrightarrow, "\text{adjoint}"] && {} \end{tikzcd}
	\begin{tikzcd}  \IntHom{\TimeB}{\Omega} \arrow[d, "\Transfer{\Phi}^\flat"] \\ \IntHom{ \TimeB }{ \IntHom{\TimeB}{\Omega} } \end{tikzcd} \]
	Although these two morphisms $\Transfer{\Phi}^\flat$ and $\shift^\flat$ are  between the same pair of objects, they describe two completely different types of dynamics. These morphisms come together in the following more general result.
	
	\paragraph{Equalizing dynamics} Let $X$ be an object in $\Context$ and $F, G$ be two pre-flows on $X$ as shown below on the left. Shown on the right are their respective right adjoints.
	\begin{equation} \label{eqn:eqlz_dyn:1}
		\begin{tikzcd} \TimeB \otimes X \arrow[d, "F"'] \\ X \end{tikzcd} ,\;
		\begin{tikzcd} \TimeB \otimes X \arrow[d, "G"'] \\ X \end{tikzcd} \imply
		\begin{tikzcd} X \arrow[d, "F^\flat"'] \\ \IntHom{\TimeB}{X} \end{tikzcd} ,\;
		\begin{tikzcd} X \arrow[d, "G^\flat"'] \\ \IntHom{\TimeB}{X} \end{tikzcd}
	\end{equation}
	The latter two morphisms have an equalizer :
	\begin{equation} \label{eqn:eqlz_dyn:2}
		\begin{tikzcd}
			& \IntHom{ \TimeB }{ X } \\ 
			\IntHom{ \TimeB }{ X } \arrow[ur, "=", bend left=20] \arrow[dr, "="', bend right=20] & E \arrow{r}{q} \arrow[l, dashed] & X \arrow[bend right=20]{ul}[swap]{ F^\flat } \arrow[bend left=20]{dl}{ G^\flat } \\
			& \IntHom{ \TimeB }{ X } 
		\end{tikzcd}
	\end{equation}
	We can say the following :
	
	\begin{theorem} \label{thm:eqlz_dyn}
		Suppose Assumptions \ref{A:monoidal}, \ref{A:closed} and \ref{A:enrich_fnctr} hold. Assume the morphisms and objects in \eqref{eqn:eqlz_dyn:1} and \eqref{eqn:eqlz_dyn:2}. Then there exists a unique morphism $F\cap G : \TimeB \otimes E  \to E$
		such that the following diagram commutes:
		\begin{equation} \label{eqn:eqlz_dyn:3}
			\begin{tikzcd} 
				\TimeB\otimes X \arrow[d, "F"] &&& \TimeB\otimes E \arrow{d}{F\cap G} \arrow[rrr, "\TimeB\otimes q"] \arrow[lll, "\TimeB\otimes q"'] &&& \TimeB\otimes X \arrow[d, "G"'] \\
				X &&& E \arrow[rrr, "q"] \arrow[lll, "q"'] &&& X
			\end{tikzcd}
		\end{equation}
		Moreover, if $F,G$ are left-actions of $\Time$, then so is $F\cap G$.
	\end{theorem}
	
	Theorem \ref{thm:eqlz_dyn} is proved in Section \ref{sec:proof:eqlz_dyn}.
	The morphism $F\cap G$ so constructed is an equalizer of the dynamics under the pre-flows $F,G$ respectively. Any pre-flow has an alternate and equivalent description as a morphism from the state-space $X$ into the path-space $\IntHom{\TimeB}{X}$. The pre-flow $F\cap G$ has as its domain that subspace of $X$ which generates the same paths under both $F$ and $G$. The pre-flow $F\cap G$ thus encodes the commonality between $F$ and $G$, in terms of the paths they produce from each state.
	
	The most interesting application of Theorem \ref{thm:eqlz_dyn} is to take :
	\[X = \IntHom{\TimeB}{\Omega}, \quad F = \shift_{\Omega}, \quad G = \Transfer{\Phi}.\]
	We are interested in the morphism shown below on the left
	\begin{equation} \label{eqn:thm:3:1}
		\begin{tikzcd} E_{\Phi} \arrow{d}[swap]{q_{\Phi}} \\ \IntHom{ \TimeB }{ \Omega } \end{tikzcd}
		\;:=\; 
		\begin{tikzcd}
			& \IntHom{ \TimeB }{ \IntHom{\TimeB}{\Omega} } \\ 
			\IntHom{ \TimeB }{ \IntHom{\TimeB}{\Omega} } \arrow[ur, "=", bend left=20] \arrow[dr, "="', bend right=20] & E_{\Phi} \arrow{r}{q_{\Phi}} \arrow[l, dashed] & \IntHom{ \TimeB }{ \Omega } \arrow[bend right=20]{ul}[swap]{ \Transfer{\Phi}^\flat } \arrow[bend left=20]{dl}{ \shift^\flat } \\
			& \IntHom{ \TimeB }{ \IntHom{\TimeB}{\Omega} } 
		\end{tikzcd}
	\end{equation}
	which is an equalizer of $\Transfer{\Phi}^\flat$ and $\shift^\flat$. This equalizer action is shown on the right above. An equalizer is always injective. This makes $E_{\Phi}$ a subobject of the path space $\PathF (\Omega) = \IntHom{ \TimeB }{ \Omega }$. 
	
	\begin{figure} [!t]
		\center
		\begin{tikzpicture}[scale=0.7, transform shape]
			\node [style={rect4}, scale=1.7] (n1) at (-0.3\columnA, 0.5\rowA) { An action $\tilde{\Phi}$ of the semigroup $\Time$ on domain $\Omega$ };
			\node [style={rect2}, scale=0.8] (n2) at (0.5\columnA, 2.4\rowA) { Lemma \ref{lem:enrch:1}, \eqref{eqn:enrch_fnctr_time:1}  };
			\node [style={rect2}, scale=0.8] (n8) at (0.8\columnA, 1.8\rowA) { Assumption \ref{A:enrich_fnctr} -- Enriched functor  };
			\node [style={rect3}, scale=1.1] (n3) at (1.8\columnA, 3\rowA) { A subgroup $\Phi^\sharp : \TimeB \to \Endo( \Omega )$ of endomorphisms \eqref{eqn:def:hatPhi}};
			\node [style={rect2}, scale=0.8] (n4) at (3\columnA, 1.5\rowA) { Lemma \ref{lem:enrch:1}, \eqref{eqn:enrch_fnctr_time:2}  };
			\node [style={rect2}, scale=0.8] (n9) at (2.4\columnA, 0.4\rowA) { Theorem \ref{thm:8} - functorial };
			\node [style={rect3}, scale=1.1] (n5) at (2.7\columnA, -1.4\rowA) { A left-action $\Phi : \TimeB \otimes \Omega \to \Omega$ of $\TimeB$ on $\Omega$ \eqref{eqn:def:tildePhi} };
			\node [style={rect2}, scale=0.8] (n6) at (2.3\columnA, -3\rowA) { Theorem \ref{thm:3} };
			\node [style={rect2}, scale=0.8] (n10) at (0.1\columnA, -1.2\rowA) { Theorem \ref{thm:6} - functorial};
			\node [style={rect4}, scale=1.7] (n7) at (1.1\columnA, -2.9\rowA) { The action of $\tilde{\Phi}$ embedded as a flow in the path space $\PathF (\Omega)$ };
			\draw[-to] (n1) to (n2);
			\draw[-to] (n2) to (n3);
			\draw[-to] (n1) to (n8);
			\draw[-to] (n8) to (n3);
			\draw[-to] (n3) to (n4);
			\draw[-to] (n4) to (n5);
			\draw[-to] (n3) to (n9);
			\draw[-to] (n9) to (n5);
			\draw[-to] (n5) to (n6);
			\draw[-to] (n6) to (n7);
			\draw[-to] (n1) to (n10);
			\draw[-to] (n10) to (n7);
		\end{tikzpicture}
		\caption{The subshift functor. As shown in Figure \ref{fig:dyn_interpret}, the most basic description of a  dynamical system is as the action of a semigroup $\Time$ on a domain $\Omega$.  In the usual contexts of set-theoretic or topological dynamics, it is well known that each state $\omega\in \Omega$ leads to a trajectory under the dynamics. In a general category-$\Context$ these trajectories or paths together form a new space $\PathF(\Omega)$ called the \emph{path-space}, which is again a member of category-$\Context$. Each trajectory is equipped with a natural motion along it. These motions collectively create a dynamics on $\PathF(\Omega)$. Thus overall one derives a new dynamical system from the existing dynamics. We describe this construction for a general category $\Context$, which generalizes the set-theoretic and topological cases. We also show that this construction is functorial. }
		\label{fig:subshift_functor}
	\end{figure} 
	
	\begin{theorem}[Subshift representation of a dynamical system] \label{thm:3}
		Suppose Assumptions \ref{A:dyn}, \ref{A:monoidal} and \ref{A:closed} hold, and $\Context$ has equalizers. Assume the notation of \eqref{eqn:thm:3:1}. Then
		\begin{enumerate} [(i)]
			\item there is a equalizer 
			\[q : E_{\Phi} \to \PathF(\Omega)\]
			of the morphisms $\shift_{\Omega}^\sharp$, $\Transfer{\Phi}^\sharp$ from $\PathF(\Omega)$ to $\PathF \paran{ \PathF(\Omega) }$.
			\item $E_{\Phi}$ is a subobject of $\PathF(\Omega)$.
			\item There exists a unique morphism 
			\begin{equation} \label{eqn:thm:3:2}
				\shiftF{\Phi} : \TimeB \otimes E_{\Phi}  \to E_{\Phi} ,
			\end{equation}
			such that the following diagram commutes:
			\begin{equation} \label{eqn:thm:3:3}
				\begin{tikzcd}
					E_{\Phi} \arrow{d}{q} && \TimeB \otimes E_{\Phi} \arrow{d}[swap]{\TimeB \otimes q} \arrow{rr}{ \shiftF{\Phi} } \arrow{ll}[swap]{ \shiftF{\Phi} } && E_{\Phi} \arrow{d}{q} \\
					\IntHom{ \TimeB }{ \Omega } && \TimeB \otimes \IntHom{ \TimeB }{ \Omega } \arrow{ll}[swap]{ \tilde{\Transfer{\Phi}} } \arrow{rr}{ \shift } && \IntHom{ \TimeB }{ \Omega }
				\end{tikzcd}
			\end{equation}
			\item The morphism $\shiftF{\Phi}$ in \eqref{eqn:thm:3:2} is a left action of $\TimeB$.
			\item Suppose $f:\Phi \to \Psi$ is a morphism in the category of left $\TimeB$-actions (Definition \ref{desc:dyn:6}). Then there is a unique morphism $E(f):E_{\Phi} \to E_{\Psi}$ such that
			\begin{equation} \label{eqn:thm:3:4}
				\begin{tikzcd}
					E_{\Phi} \arrow{d}[swap]{q} \arrow{rr}{E(f)} && E_{\Psi} \arrow{d}{q'} \\
					\PathF(\Omega) \arrow{rr}[swap]{ \IntHom{\TimeB}{f} } && \PathF(\Omega')
				\end{tikzcd} , \quad 
				\begin{tikzcd}
					\TimeB \otimes E_{\Phi} \arrow{rr}{\TimeB \otimes E(f)} \arrow{d}[swap]{\shift_\Phi} && \TimeB \otimes E_{\Psi} \arrow{d}{\shift_\Psi} \\
					E_{\Phi} \arrow{rr}{E(f)} && E_{\Psi}
				\end{tikzcd}
			\end{equation}
			where $q$ and $q'$ are universal morphisms of the equalizer.
		\end{enumerate}
	\end{theorem}
	
	Theorem \ref{thm:3} is proved in Section \ref{sec:proof:thm:3}. Every trajectory of $\tilde{\Phi}$ is a point in the space $\PathF(\Omega)$. However not every point in $\PathF(\Omega)$ is a trajectory of $\tilde{\Phi}$. Theorem \ref{thm:3} explicitly reconstructs the dynamics of $\tilde{\Phi}$'s orbit space as an embedded dynamics within $\PathF(\Omega)$. This essence of Theorem \ref{thm:3} was distilled earlier in Corollary \ref{corr:subshift}. The flow $\shiftF{\Phi}$ from \eqref{eqn:thm:3:2} thus leads to a semigroup action
	\begin{equation} \label{eqn:thm:3:5}
		\shiftF{\Phi} : \Time \to \Endo \paran{ E_{\Phi} }.
	\end{equation}
	The surprising revelation of Theorem \ref{thm:3} that goes beyond Theorem \ref{thm:eqlz_dyn} is that the correspondence of a flow with its subshift is also functorial. In other words the  operation $E$ is an endofunctor on the category of left actions. This functor will be revealed to be even more informative. We next show that the semigroup action in \eqref{eqn:thm:3:4} is an isomorphic copy of the semigroup action of $\tilde{\Phi}$. 	
	
	\begin{theorem} \label{thm:6}
		Suppose Assumptions \ref{A:dyn}, \ref{A:monoidal}, \ref{A:closed} and \ref{A:enrich_fnctr} hold, and additionally, and $\TimeB$ be a commutative monoid. Let  $\Phi^\flat : \Omega \to \IntHom{ \TimeB }{ \Omega }$ be the adjoint of a $\TimeB$-action $\Phi : \TimeB \otimes \Omega \to \Omega$. Then $\Phi^\flat$ coincides with the equalizer $q$ from \eqref{eqn:thm:3:1}. Consequently the subshift domain $E_{\Phi}$ is isomorphic to $\Omega$.    
	\end{theorem}
	
	Theorem \ref{thm:6} is proved in Section \ref{sec:proof:thm:6}. One of the major revelations of Theorem \ref{thm:6} was announced earlier in Corollary \ref{corr:dom_subshift}. A crucial property of the time-0 evaluation morphism \eqref{eqn:def:path_eval:1} is as a left inverse to the flow :
	\begin{equation} \label{eqn:def:path_eval:3}
		\begin{tikzcd}
			\Omega \arrow[dr, "\Phi^\flat"'] \arrow[rr, "="] & & \Omega \\
			& \IntHom{\TimeB}{\Omega} \arrow[ur, "\eval_0"']
		\end{tikzcd}
	\end{equation}
	This commutation is a direct consequence of the unital identity \eqref{eqn:enrch_fnctr_time:2} of the flow $\Phi$. Theorem \ref{thm:3}, Theorem \ref{thm:6} and \eqref{eqn:def:path_eval:3} combine to give
	\[\begin{tikzcd}
		\TimeB \otimes E_{\Phi} \arrow[dr, "\cong"] \arrow[dd, "\TimeB \otimes q"'] \arrow[rrr, "\shiftF{\Phi}"] & & & E_{\Phi} \arrow[dl, "\cong"] \arrow[dd, "q"] \arrow[r, "\cong"] & \Omega \\
		& \TimeB \otimes \Omega \arrow[dl, "\TimeB \otimes \Phi^\flat"] \arrow[r, dotted] & \Omega \arrow[dr, "\Phi^\flat"] \\
		\TimeB \otimes \IntHom{\TimeB}{\Omega} \arrow[rrr, "\shiftF{\Omega}"] & & & \IntHom{\TimeB}{\Omega} \arrow[uur, bend right=10, "\eval_0"']
	\end{tikzcd}\]
	This indicates that irrespective of $\tilde{\Phi}$ and its nature as a flow or semi-flow, the subshift representation retains an isomorphic copy of $\Omega$ at each time instant.
	
	This completes our investigation of the theoretical outcomes of the four ground assumptions - Assumptions \ref{A:dyn}, \ref{A:monoidal}, \ref{A:closed} and \ref{A:enrich_fnctr}. These structural assumptions enable not only the primary description of a dynamical system as a semigroup action, but also provides additional structure $\TimeB$ to the time semigroup $\Time$, provides three other equivalent descriptions of mathematics each capturing a different role, enables the notions of paths, path-space and subshifts. We next examine how it also allows the notion of states or points and state-preservation to be incorporated within the categorical framework.	
	
	%-_-_-_-_-_-_-_-_-_-_-_-_-_-_-_-_-_-_-_-_-_-_-_-_-_-_-_-_-_-_-_-_-_-_-_-_-_-_-_-_-_-_-_-_-_-_-_-_-_-_-_-_-_-_-_-_-_-_-_-_-_-_-_-_-_-_-_-_-_-_-_-_-_-_-_-
	\section{States of a dynamics functor} \label{sec:states}
	
	Dynamical systems theory is primarily interested in the asymptotic properties of trajectories, and on invariant points, subsets and measures. All these notions are built upon the primitive notions of states and invariance. In this section we examine how the notion of states arise naturally in the enriched category framework. An important assumption that enables a complete discussion of states is
	
	\begin{Assumption} \label{A:trmnl}
		The unit element $1_{\Context}$ is terminal.
	\end{Assumption}
	
	Recall that being \emph{terminal} means that for every object $x\in \Context$, there is a unique morphism from $x$ to $1_{\Context}$. While Assumption \ref{A:trmnl} is required for general monoidal categories, it automatically holds true for Cartesian categories, as implied by the following lemma.
	
	\begin{lemma}
		Let $\calC$ be a Cartesian closed category. Then an object is the unit object iff it is the terminal object.
	\end{lemma}
	
	With this Assumption in mind, we define :
	
	\begin{definition} [States]
		A \emph{state} of an object $\Omega$ in a monoidal category $\Context$ is any morphism $\omega : 1_{\Context} \to \Omega$.
	\end{definition}
	
	This concept of state is similar to the concept of an \emph{element} \cite[e.g.]{Flori2013Topos, Barr1986} in a general category, in which a monoidal product is replaced by a categorical product, and the monoidal unital element $1_{\Context}$ by a terminal element. In our context $\Context$, we thus interpret a state to be any morphism from the distinguished element $1_{\Context}$. The unit cum terminal object makes possible an elementary functor, which is of fundamental importance :
	
	\begin{definition} [Terminal functor]
		If Assumption \ref{A:trmnl} holds then there is an enriched functor $\Delta : \Time \to \Context$ with constant value $1_{\Context}$. The connecting morphism as defined in \eqref{eqn:def:enrich_func:Time:1} is the unique terminal morphism shown below on the left :
		\[ \begin{tikzcd} \TimeB \arrow[r, "!", "\Delta^\sharp"'] & 1_{\Context} \end{tikzcd} \]
		and the commutation relation \eqref{eqn:def:enrich_func:Time:2} trivially holds, as each of the objects in the second row become $1_{\Context}$.
	\end{definition}
	
	The connecting morphism $\Delta^\sharp$ satisfies the following commutation 
	\begin{equation} \label{eqn:ibx9z3}
		\forall
		\begin{tikzcd} 1_{\Context} \arrow[d, "\alpha"'] \\\IntHom{1_{\Context}}{\Omega} \end{tikzcd}, \quad
		\begin{tikzcd}
			\TimeB \arrow{d}[swap]{\cong} \arrow[r, "!"] & 1_{\Context} \arrow[r, "\alpha"] & \IntHom{1_{\Context}}{\Omega} \\
			1_{\Context} \otimes \TimeB \arrow{rr}[swap]{ \alpha \otimes \Delta^\sharp } && \IntHom{1_{\Context}}{\Omega} \otimes 1_{\Context}\arrow{u}[swap]{ \cong }
		\end{tikzcd}
	\end{equation}
	This leads to one of our new constructs :
	
	\begin{definition} [Enriched stationary state]
		An enriched stationary state of a functor $\tilde{\Phi}$ as in Assumption \ref{A:enrich_fnctr} is an enriched natural transformation $\tilde{\omega} : \Delta \Rightarrow \tilde{\Phi}$.
	\end{definition}
	
	An enriched natural transformation $\tilde{\omega} : \Delta \Rightarrow \tilde{\Phi}$ as above has only one connecting morphism which is
	\begin{equation} \label{eqn:def:enrich_nat:cone:1}
		\omega_* : 1_{\Context} \to \IntHom{1_{\Context}}{\Omega},
	\end{equation}
	Thus the enriched natural transformation $\tilde{\omega}$ leads to an element $\omega_*$ of the internal-element object $\IntHom{1_{\Context}}{\Omega}$. Since $\Context$ is closed, $\omega_*$ has a left-dual, which is
	\begin{equation} \label{eqn:def:enrich_nat:cone:4}
		^{*}\paran{ \omega_* } : 1_{\Context} \otimes 1_{\Context} \to \Omega.
	\end{equation}
	This morphism can be trivially combined with the unitor isomorphism for $1_{\Context}$ to get
	\begin{equation} \label{eqn:def:enrich_nat:cone:5}
		\begin{tikzcd}
			1_{\Context} \arrow[Shobuj, d, "\omega"', dashed] \arrow[rr, "\lambda_{1}^{-1}", "\cong"'] && 1_{\Context} \otimes 1_{\Context} \arrow[dll, "^{*}\paran{ \omega_* }"] \arrow[bend left=10, dr, "1_{\Context} \otimes \omega_*"] \\
			\Omega && & 1_{\Context} \otimes \IntHom{1_{\Context}}{\Omega} \arrow[lll, pos=0.2, "\eval_{1,\Omega}"']
		\end{tikzcd}
	\end{equation}
	The right commuting triangle follows from Lemma \ref{lem:idl3v:1} which will be presented later. It expresses a left dual in terms of the evaluation morphism. Thus an enriched stationary state leads to an ordinary state $\omega$ of $\Omega$.
	
	\begin{theorem} [Characterization of stationary states] \label{thm:sttnry}
		Suppose Assumptions \ref{A:dyn}, \ref{A:monoidal}, \ref{A:closed}, \ref{A:enrich_fnctr} and \ref{A:trmnl} hold. Then the following are equivalent :
		\begin{enumerate} [(i)]
			\item $\tilde{\Phi}$ has an enriched stationary state $\tilde{\omega}$. %, i.e., an enriched natural transformation  $\tilde{\omega} : \Delta \Rightarrow \tilde{\Phi}$.
			\item There is a morphism $\omega_*$ as in \eqref{eqn:def:enrich_nat:cone:1} which satisfies 
			\begin{equation} \label{eqn:def:enrich_nat:cone:2}
				\begin{tikzcd}
					\TimeB \otimes 1_{\calC} \arrow{rr}{ \Phi^\sharp \otimes \omega_* } && \Endo(\Omega) \otimes \IntHom{1_{\Context}}{\Omega} \arrow{d}{ \circ } \\
					\TimeB \arrow{u}{\cong} \arrow{d}[swap]{\cong} && \IntHom{1_{\Context}}{\Omega} \\
					1_{\calC} \otimes \TimeB \arrow{rr}[swap]{ \tilde{\omega}_{*} \otimes \Delta^\sharp } && \IntHom{1_{\Context}}{\Omega} \otimes 1_{\Context}\arrow{u}[swap]{ \cong }
				\end{tikzcd}
			\end{equation}
			\item There is a morphism $\omega_*$ as in \eqref{eqn:def:enrich_nat:cone:1} which satisfies 
			\begin{equation} \label{eqn:def:enrich_nat:cone:3}
				\begin{tikzcd}
					\TimeB \otimes 1_{\calC} \arrow{rr}{ \Phi^\sharp \otimes \omega_* } && \Endo(\Omega) \otimes \IntHom{1_{\Context}}{\Omega} \arrow{d}{ \circ } \\
					\TimeB \arrow{u}{\cong} \arrow[r, "!"] & 1_{\Context} \arrow[r, "\omega_*"] & \IntHom{1_{\Context}}{\Omega} 
				\end{tikzcd}
			\end{equation}
		\end{enumerate}
		Any of these equivalent conditions imply
		\begin{enumerate} [(i), resume]
			\item There is a morphism $\omega_*$ as in \eqref{eqn:def:enrich_nat:cone:1} whose dual $\omega$ as in \eqref{eqn:def:enrich_nat:cone:5} satisfies
			\begin{equation} \label{eqn:jd92}
				\begin{tikzcd}
					\TimeB \otimes 1_{\Context} \arrow[r, "!"] \arrow[d, "\TimeB \otimes \omega"'] & 1_{\Context} \arrow[d, "\omega"] \\
					\TimeB \otimes \Omega \arrow[r, "\Phi"'] & \Omega
				\end{tikzcd}
			\end{equation}
		\end{enumerate}
	\end{theorem}
	
	Theorem \ref{thm:sttnry} is proved in Section \ref{sec:proof:sttnry}. Equation \eqref{eqn:jd92} is the classical representation of a stationary state of a pre-flow $\Phi$. One of the consequence of Theorem \ref{thm:sttnry} is that an enriched stationary state is also a stationary state in the ordinary sense. Such an implication is desirable since our objective is to establish enriched functors as an extension of the notion of ordinary dynamical systems. Being an extension should mean that the associated notions such as states as semi-group actions are also extended.
	
	The notion of state leads to the notion of orbits. We utilize Theorem \ref{thm:8} which makes the various descriptions of dynamics interchangeable. We shall be using the flow-form $\Phi : \TimeB \otimes \Omega \to \Omega$ of $\tilde{\Phi}$. Consider any state $\omega$ of $\Omega$ as shown below on the left :
	\begin{equation} \label{eqn:def:orbit}
		\begin{tikzcd} 1_{\Context} \arrow{d}{ \omega } \\ \Omega \end{tikzcd} \imply
		\begin{tikzcd}
			\TimeB\otimes 1_{\Context} \arrow{d}[swap]{ \TimeB \otimes \omega } & \TimeB \arrow[l, "="'] \arrow[d, dashed, "\Orbit_{\omega}"] \\
			\TimeB \otimes \Omega \arrow{r}{\Phi} & \Omega
		\end{tikzcd} 
		\begin{tikzcd} {} \arrow[rr, "\text{adjoint}"] && {} \end{tikzcd}
		\begin{tikzcd} 1_{\Context} \arrow{d}{ \Orbit_{\omega}^\flat } \\ \PathF(\Omega)  \end{tikzcd}
	\end{equation}
	This leads to the composite morphism $\Orbit_{\omega}$ shown in the center. Its right adjoint $\Orbit_{\omega}^\flat$ is that element of $\PathF(\Omega)$ which corresponds to the path traced out by $\omega$. 
	
	Although our construction of the path-space object $\PathF (\Omega)$ has been categorical, it contains all the information one would associate to the classical definition of a path-space. Unlike a set-theoretic description, this information is encoded not within $\PathF (\Omega)$ but in its outgoing morphisms. For example, the time-0 evaluation which associates to every path its current state, is created as the following composite morphism :
	\begin{equation} \label{eqn:def:path_eval:1}
		\begin{tikzcd}
			1_{\Context} \otimes \PathF (\Omega) \arrow{d}[swap]{ \start \otimes \PathF (\Omega) } & \PathF (\Omega) \arrow[l, "="'] \arrow[d, "\eval_0", dashed, Shobuj] \\
			\TimeB \otimes \PathF (\Omega) \arrow[r, "\eval"] & \Omega
		\end{tikzcd}
	\end{equation}
	The time-0 evaluation can be combined with the natural flow of the path to get the time-t evaluation operation :
	\begin{equation} \label{eqn:def:path_eval:2}
		\begin{tikzcd}
			1_{\Context} \otimes \PathF (\Omega) \arrow{d}[swap]{ \start \otimes \PathF (\Omega) } & \PathF (\Omega) \arrow[l, "="'] \arrow[d, "\eval_0", dashed] && \TimeB \otimes \PathF (\Omega) \arrow[ll, "\shift_{\Omega}"'] \\
			\TimeB \otimes \PathF (\Omega) \arrow[r, "\eval"] & \Omega & \PathF (\Omega) \arrow[l, dashed, "\eval_t", Shobuj] \arrow[r, "="] & 1_{\Context} \otimes \PathF (\Omega) \arrow[u, "t\otimes \PathF (\Omega)"']
		\end{tikzcd}
	\end{equation}
	This completes our presentation of a categorical language for dynamical systems that completely captures all different aspects of dynamics - as a flow, as a semigroup action, as a parameterized family of endomorphisms, and as a subshift of the path-space. We have proposed that dynamical systems be studied as enriched functors and not simply as functors of a semigroup into the semigroup of endomorphisms of a domain $\Omega$. This interpretation is enabled by four categorical Assumptions \ref{A:dyn}, \ref{A:monoidal}, \ref{A:closed} and \ref{A:enrich_fnctr}.
	The preliminary notions of a dynamical system, such as states, invariance and paths occur naturally as by-products. Theorem \ref{thm:8} shows how this makes the three descriptions of dynamical systems interchangeable. This leads to the concept of an states, stationary enriched states, and stationary states. Theorem \ref{thm:8} and the notion of states lead to the notion of an orbit as an element of the path space of an object. 
	
	This work aims to provides the platform for a deeper investigation into dynamics, using a completely categorical language. In our final section, we look at  some further challenges to applying this framework to the classical contexts in dynamical systems theory. 
	
	%-_-_-_-_-_-_-_-_-_-_-_-_-_-_-_-_-_-_-_-_-_-_-_-_-_-_-_-_-_-_-_-_-_-_-_-_-_-_-_-_-_-_-_-_-_-_-_-_-_-_-_-_-_-_-_-_-_-_-_-_-_-_-_-_-_-_-_-_-_-_-_-_-_-_-_-
	\section{Conclusions} \label{sec:conclus}
	
	As mentioned earlier, the various contexts in dynamical systems usually have the neat structure of a functor category. This is illustrated by the following chain of monoidal categories :
	\begin{equation} \label{eqn:chain:1}
		\begin{tikzcd} [scale cd = 0.8]
			\VectCat \arrow[r, "\iota"] & \ManCat{r} \arrow[r, "\iota'"] & \Topo \arrow[r, "\iota''"] & \MeasCat \arrow[r, "\iota'''"] & \StochCat
		\end{tikzcd}
	\end{equation}
	The four context categories chosen in this diagram satisfy Assumption \ref{A:monoidal} and \ref{A:closed}. If time $\Time$ is chosen to be a topological semigroup such as $\Nplus$, $\integer$ or $\real$ then Assumptions \ref{A:dyn} and \ref{A:monoidal} would also be satisfied. The sequence of inclusions imply the fact that linear spaces and transformations are also smooth spaces and transformations; smooth spaces and transformations are topological spaces and continuous transformations; continuous transformations are Borel-measurable spaces and measurable transformations; and the latter are measurable spaces and Markov transitions. The last link is based on the realization that a deterministic map $f:X\to Y$ between two spaces is also a Markov transition, associating the point-measure $\delta_{f(x)}$ to each point $x\in X$. 
	The diagram \eqref{eqn:chain:1} give the following chain of dynamical systems
	\begin{equation} \label{eqn:chain:3}
		\begin{tikzcd} [scale cd = 0.8]
			\begin{array}{c} \mbox{Linear} \\ \mbox{dynamics} \\\Functor{\Time}{ \VectCat } \end{array}\arrow[r, "\iota\circ"] & 
        \begin{array}{c} C^r-\mbox{smooth} \\ \mbox{dynamics} \\ \Functor{\Time}{ \ManCat{r} } \end{array} \arrow[r, "\iota'\circ"] & 
        \begin{array}{c} \mbox{Topological} \\ \mbox{dynamics} \\  \Functor{\Time}{ \Topo } \end{array} \arrow[r, "\iota''\circ"] & 
        \begin{array}{c} \mbox{Measurable} \\ \mbox{dynamics} \\ \Functor{\Time}{ \MeasCat } \end{array} \arrow[r, "\iota'''\circ"] & 
        \begin{array}{c} \mbox{Markov} \\ \mbox{process} \\ \Functor{\Time}{ \StochCat } \end{array}
		\end{tikzcd} 
	\end{equation}
	The diagram in \eqref{eqn:chain:3} is useful as it formally states an inclusion of categories of dynamical systems in various categories. However this diagram only involves ordinary functor categories, and not enriched functor categories. So although intuitive and comprehensive, our theoretical results on enriched categories do not apply to these functors. Ideally, the dynamics in two different contexts should be expressed as enriched functor categories connected via enriched functors. There are two main challenges to this. First is the nature of the context category itself. The categories $\Topo$, $\MeasCat$ and $\StochCat$ are monoidal but not closed. There are many alternatives to these classical categories :
	\begin{enumerate} [(i)]
		\item the category $\CGWHS$ of \emph{compactly generated, weakly generated Hausdorff spaces} \cite{Steenrod1967cnvnnt} is a Cartesian closed subcategory of $\Topo$.
		\item the category $\QuasiBorel$ \cite{HeunenEtAl2017cnvnt} of \emph{quasi-Borel spaces} is a Cartesian closed subcategory of $\MeasCat$.
		\item the category $\Diffeological$ of \emph{diffeological spaces} \cite{Stacey2008smooth, BaezHoffnung2011cnvnnt} is a Cartesian closed subcategory of $\ManCat{1}$. Another alternative to $\ManCat{1}$ are Chen spaces \cite{Chen1977iterated}.
	\end{enumerate}
	An important task ahead is to develop the theory of enriched $\Time$-domain functors on these closed categories, and relate them to contexts they are meant to replace. 
	
	A field where a categorical formulation of dynamics can have a major impact is in the study of numerical methods for reconstructing dynamical systems. This field was recently placed in a categorical footing by the authors in \cite{DasSuda2024recon}. The entire process of reconstruction was presented as an inter-play between four different categories -- observed dynamical systems, dynamical systems, timeseries data, and subshifts. Reconstructions were shown to be closely related to left adjoints. One of the classic results in ergodic theory is that any finite timeseries can be created from Bernoulli processes as well as deterministic processes. To study the convergence of reconstruction algorithms, one needs to be cognizant of the marginal distributions of the source of the data. This requires incorporating a categorical version of the Kolmogorov consistency theorem \cite[e.g.]{FritzRischel2020inf, vanBelle2023mrtngl} using the language of \emph{filtered colimits} \cite[e.g.]{AndrekaNemeti1982direct}.
	
	Developing dynamical systems theory in a categorical language also requires new proofs of many of the classical results in this theory. An important example is a categorical proof of the ergodic decomposition theorem \cite{MossPerrone2022ergdc}. An important and unfinished task is show that the infinitely and densely nested character of chaotic, ergodic systems is a natural consequence of underlying categorical assumption. While the notion of nesting corresponds to morphisms and chains of morphisms in $\Functor{\Time}{\Topo}$, the concept of density is a harder concept to capture using categorical language \cite[e.g.]{ClementinoTholen1997sep, ClementinoGiuliTholen1996}.
	
	Developing a categorical restatement of dynamical systems theory is a promising venture, and there still remain several challenges such as these.
	
	\section*{Acknowledgments}
	T.S. is grateful to the Grants-in-Aid for Scientific Research (Kakenhi) 25K17284.
	
	\section*{Declarations} The authors have no competing interests, financial, academic or otherwise. No AI-tools were used in the preparation of the manuscript. No data was used in this research.
	%-_-_-_-_-_-_-_-_-_-_-_-_-_-_-_-_-_-_-_-_-_-_-_-_-_-_-_-_-_-_-_-_-_-_-_-_-_-_-_-_-_-_-_-_-_-_-_-_-_-_-_-_-_-_-_-_-_-_-_-_-_-_-_-_-_-_-_-_-_-_-_-_-_-_-_-
	\section{Appendix}
	%........................................................................................................................
	\subsection{Some technical lemmas} \label{lem:technical}
	
	\begin{lemma} \label{lem:idl3v:1}
		Suppose Assumptions \ref{A:monoidal} and \ref{A:closed} hold. Consider a morphism of the form $\phi : \TimeB \otimes X \to \Omega$ and its right-adjoint $\phi^* : X \to \IntHom{\TimeB}{\Omega}$. Then one has the following commutation 
		\begin{equation} \label{eqn:idl3v:1}
			\begin{tikzcd}
				\TimeB \otimes X \arrow{d}[swap]{ \TimeB \otimes \phi^* } \arrow{r}{\phi} & \Omega \\
				\TimeB \otimes \IntHom{\TimeB}{\Omega} \arrow{ur}[swap]{ \eval_{\Omega} }
			\end{tikzcd}
		\end{equation}
	\end{lemma}
	
	Lemma \ref{lem:idl3v:1} is a direct consequence of the fact that the evaluation morphism corresponds to the counit of the product-inner product adjunction of $\Context$. Lemma \ref{lem:idl3v:1} provides a formulation for left-adjoint $^*\gamma$ of a morphisms $\gamma$ with co-domain $\IntHom{\TimeB}{\Omega}$ :
	\begin{equation} \label{eqn:idl3v:4}
		\begin{tikzcd}
			\TimeB \otimes X \arrow{d}[swap]{ \TimeB \otimes \gamma } \arrow[dashed]{r}{^*\gamma} & \Omega \\
			\TimeB \otimes \IntHom{\TimeB}{\Omega} \arrow{ur}[swap]{ \eval_{\Omega} }
		\end{tikzcd}
	\end{equation}
	Thus the left adjoints of morphisms with co-domain $\IntHom{\TimeB}{\Omega}$ can be obtained by composition with the evaluation morphism. Two special cases of Lemma \ref{lem:idl3v:1} are important to us. The first is the following instance of $\tilde{\Phi}$ : 
	\[ X:= \TimeB \otimes \IntHom{ \TimeB }{ \Omega } , \quad \phi := \TimeB \otimes \TimeB \otimes \IntHom{ \TimeB }{ \Omega } \xrightarrow{\add \otimes \IntHom{ \TimeB }{ \Omega } }  \TimeB \otimes \IntHom{ \TimeB }{ \Omega } \xrightarrow{\eval_\Omega} \Omega. \]
	Recall that the adjoint of this morphism is the $\shift$-morphism. Then the commutation in \eqref{eqn:idl3v:1} takes the form :
	\begin{equation} \label{eqn:idl3v:2}
		\begin{tikzcd}
			\TimeB \otimes \TimeB \otimes \IntHom{ \TimeB }{ \Omega } \arrow{d}[swap]{ \TimeB \otimes \shift } \arrow{rr}{ \add \otimes \IntHom{ \TimeB }{ \Omega } } && \TimeB \otimes \IntHom{ \TimeB }{ \Omega } \arrow{d}{ \eval_{\Omega} } \\
			\TimeB \otimes \IntHom{\TimeB}{\Omega} \arrow{rr}[swap]{ \eval_{\Omega}  } && \Omega
		\end{tikzcd}
	\end{equation}
	The next instance of Lemma \ref{lem:idl3v:1} assumes
	\[ X:= \TimeB \otimes \IntHom{ \TimeB }{ \Omega } , \quad
	\begin{tikzcd}
		\TimeB \otimes \TimeB \otimes \IntHom{ \TimeB }{ \Omega } \arrow[dr, dashed, "\phi"'] \arrow{rr}{ \swap_{T,T} \otimes \IntHom{ \TimeB }{ \Omega } } && \TimeB \otimes \TimeB \otimes \IntHom{ \TimeB }{ \Omega } \arrow{d}[swap]{ \TimeB \otimes \eval_{\Omega} } \\
		& \Omega & \TimeB \otimes \Omega \arrow{l}[swap]{ \Phi  }
	\end{tikzcd}\]
	Recall that the adjoint to this $\tilde{\Phi}$ is the morphism $\Transfer{\Phi}$ from \eqref{eqn:def:transfer}. Then the commutation in \eqref{eqn:idl3v:1} takes the form : 
	\begin{equation} \label{eqn:idl3v:3}
		\begin{tikzcd}
			\TimeB \otimes \TimeB \otimes \IntHom{ \TimeB }{ \Omega } \arrow{d}[swap]{ \TimeB \otimes \Transfer{\Phi} } \arrow{rr}{ \swap_{T,T} \otimes \IntHom{ \TimeB }{ \Omega } } && \TimeB \otimes \TimeB \otimes \IntHom{ \TimeB }{ \Omega } \arrow{d}[swap]{ \TimeB \otimes \eval_{\Omega} } \\
			\TimeB \otimes \IntHom{ \TimeB }{ \Omega } \arrow{r}{ \eval_{\Omega}  } & \Omega & \TimeB \otimes \Omega \arrow{l}[swap]{ \Phi  }
		\end{tikzcd}
	\end{equation}

	We now consider two useful consequences of \eqref{eqn:idl3v:4}. Taking $X = \TimeB \otimes \TimeB \otimes \IntHom{\TimeB}{\Omega}$ and $\gamma = \shift \circ \paran{ \TimeB \otimes \shift }$ gives
	\begin{equation} \label{eqn:proof:thm:1:1}
		\begin{tikzcd} [scale cd = 0.8]
			\TimeB \otimes \TimeB \otimes \IntHom{\TimeB}{\Omega} \arrow[r, dashed, "\gamma"] \arrow{dr}[swap]{ \TimeB \otimes \shift } & \IntHom{\TimeB}{\Omega} \\
			& \TimeB \otimes \IntHom{\TimeB}{\Omega} \arrow[u, "\shift"']
		\end{tikzcd}
		\begin{tikzcd} {} \arrow[rr, "\text{left}", "\text{adjoint}"'] &&{}	\end{tikzcd}
		\begin{tikzcd} [scale cd = 0.8]
			\TimeB \otimes \TimeB \otimes \TimeB \otimes \IntHom{\TimeB}{\Omega} \arrow[rr, dashed, "^*\gamma"] \arrow{d}[swap]{\TimeB \otimes \TimeB \otimes \shift} && \Omega\\ 
			\TimeB \otimes \TimeB \otimes \IntHom{\TimeB}{\Omega} \arrow{rr}[swap]{ \TimeB \otimes \shift } && \TimeB \otimes \IntHom{\TimeB}{\Omega} \arrow{u}[swap]{\eval_{\Omega}}
		\end{tikzcd}
	\end{equation}
	Next taking $X = \TimeB \otimes \TimeB \otimes \IntHom{\TimeB}{\Omega}$ and $\gamma = \shift \circ \paran{ \add \otimes \IntHom{\TimeB}{\Omega} }$ gives
	\begin{equation} \label{eqn:proof:thm:1:5}
		\begin{tikzcd} [scale cd = 0.8]
			\TimeB \otimes \TimeB \otimes \IntHom{\TimeB}{\Omega} \arrow[r, dashed, "\gamma"] \arrow{dr}[swap]{ \add \otimes \IntHom{\TimeB}{\Omega} } & \IntHom{\TimeB}{\Omega} \\
			& \TimeB \otimes \IntHom{\TimeB}{\Omega} \arrow[u, "\shift"']
		\end{tikzcd}
		\begin{tikzcd} {} \arrow[rr, "\text{left}", "\text{adjoint}"'] &&{}	\end{tikzcd}
		\begin{tikzcd} [scale cd = 0.8]
			\TimeB \otimes \TimeB \otimes \TimeB \otimes \IntHom{\TimeB}{\Omega} \arrow[rr, dashed, "^*\gamma"] \arrow{d}[swap]{ \TimeB \otimes \mu \otimes \IntHom{\TimeB}{\Omega} } && \Omega\\ 
			\TimeB \otimes \TimeB \otimes \IntHom{\TimeB}{\Omega} \arrow{rr}[swap]{ \TimeB \otimes \shift } && \TimeB \otimes \IntHom{\TimeB}{\Omega} \arrow{u}[swap]{\eval_{\Omega}}
		\end{tikzcd}
	\end{equation}
	%
	
	%........................................................................................................................
	\subsection{Results on enriched category theory} \label{sec:proof:enrch_fnctr_time}
	
	The purpose of this section is to examine the theoretical tools needed to switch back and forth between enriched semigroups, monoids, enriched functors and flows. We begin with an important  observation that $\TimeB$ itself is a familiar categorical construct.
	
	\begin{lemma} \label{lem:enriched:3}
		Let $\Time$ be a semigroup enriched over a monoidal category $\Context$. Then the objects and morphisms $\paran{ \TimeB, \add, \start }$ defined in Section \ref{sec:time} form a monoid object. 
	\end{lemma}
	
	The lemma directly follows from the definition and identities in \eqref{eqn:def:TimeB} and \eqref{eqn:TimeB:id}. The converse also happens to be true, that a monoid object leads to a semigroup enriched over $\Context$ :    
	
	\begin{lemma} \label{lem:enriched:4}
		Let Assumptions \ref{A:monoidal} and \ref{A:closed} hold. Let $\paran{ \TimeB, \add: \TimeB\otimes \TimeB \to \TimeB, \start : 1_{\Context} \to T }$  be a monoid in $\Context$. Then we can define a one-object category $\mathcal{T}_T$ enriched over $\Context$ as follows :
		\begin{enumerate} [(i)]
			\item The only object is denoted as $*$.
			\item The morphisms correspond to the $\Context$-morphisms $1_{\Context} \to T$. The composition $a\cdot b$ of two morphisms $a$ and $b$ is defined by 
			\[\begin{tikzcd}
				1_{\Context} \arrow[d, "\cong"] \arrow[drrr, dashed, "a \cdot b", bend left = 20] \\
				1_{\Context} \otimes 1_{\Context} \arrow[rr, "b\otimes a"] && \TimeB \otimes \TimeB \arrow[r, "\add"]  & \TimeB
			\end{tikzcd} \]
			\item The identity morphism is given by $\start$.
		\end{enumerate}
	\end{lemma}
	
	The next lemma formally states an intuitive fact, that the enriched semigroup created from a monoid object which itself is created out of an enriched semigroup, coincides with the original semigroup. 
	
	\begin{lemma} \label{lem:enriched:5}
		Let Assumptions \ref{A:dyn}, \ref{A:monoidal} and \ref{A:closed} hold. Let $\TimeB$ be the $\Context$-object created from $\Time$ by enrichment, and $\mathcal{T}_{\TimeB}$ be the semigroup created from $\TimeB$ according to Lemma \ref{lem:enriched:4}. Then $\mathcal{T}_{\TimeB}$ coincides with $\Time$.
	\end{lemma}
	
	In other words the constructions Lemma \ref{lem:enriched:3} and Lemma \ref{lem:enriched:4} are inverses of each other. Having established the dual nature of enriched semigroups, we turn to  Lemma \ref{lem:enrch:1}. 
	
	\paragraph{Proof of Lemma \ref{lem:enrch:1}} Claim~(i) follows directly from the discussion preceding the theorem. We prove the Claim~(ii). We begin with the "if" part of the statement. So we assume that there is an enriched functor $\tilde{\psi}$, whose parametric form is $\psi^\sharp$, and $\psi$ is its left adjoint. Now consider the following diagram:
	\begin{equation} \label{eqn:enrch:1:1}
		\begin{tikzcd}[column sep = large]
			& 1_{\Context} \otimes \Omega \arrow[d, "\start_* \otimes \Omega"] \arrow[ddl, bend right=20, "\start_{\Omega} \otimes \Omega"'] \arrow[ddr, bend left=20, "\lambda_{\Omega}^{(l)}", "\cong"'] \\
			& \TimeB \otimes \Omega \arrow[dr, "\Psi"] \arrow[dl, "\Psi^\sharp \otimes \Omega"'] & \\
			\IntHom{\Omega}{\Omega} \otimes \Omega \arrow[rr, "\eval_{\Omega}"] & &  \Omega
		\end{tikzcd}
	\end{equation}
	The commutation in the periphery is just the unit law of the evaluation morphism. The lower commutation triangle is the adjoint relation of Lemma \ref{lem:idl3v:1}. The left commutation triangle follows from \eqref{eqn:enrch_fnctr_time:1}. The right commutation triangle is a composite of these three commutations. It also happens to be the unit law in \eqref{eqn:enrch_fnctr_time:2}. Next consider the diagram :
	\begin{equation}\label{eqn:enrch:1:2}
		\begin{tikzcd}[scale cd = 1]
			\TimeB \otimes \TimeB \otimes \Omega \arrow[rrrrr, "\TimeB \otimes \Psi"] \arrow{drr}{\TimeB\otimes  \Psi^\sharp \otimes \Omega} \arrow{dddd}{\add \otimes \Omega} && && & \TimeB \otimes \Omega \arrow{ddl}{\Psi^\sharp \otimes \Omega} \arrow[dddd, "\eval_{\Omega}"'] \\
			{} && \TimeB \otimes \IntHom{\Omega}{\Omega} \otimes \Omega \arrow{urrr}{\TimeB\otimes  \eval_\Omega} \arrow{d}{\Psi^\sharp \otimes \IntHom{\Omega}{\Omega} \otimes \Omega} && {} \\
			{} && \IntHom{\Omega}{\Omega} \otimes \IntHom{\Omega}{\Omega} \otimes \Omega \arrow{rr}{\IntHom{\Omega}{\Omega} \otimes \eval_\Omega} \arrow{d}{\circ \otimes \Omega} && \IntHom{\Omega}{\Omega} \otimes \Omega \arrow{ddr}{\eval_\Omega} \\
			&& \IntHom{\Omega}{\Omega} \otimes \Omega \arrow{drrr}{\eval_\Omega} && {} \\
			\TimeB \otimes \Omega \arrow[rrrrr, "\Psi"] \arrow{urr}{\Psi^\sharp \otimes \Omega} && && & \Omega
		\end{tikzcd}
	\end{equation}
	Each of the commutation loops in this diagram are instances of the adjunction in Lemma \ref{lem:idl3v:1}. The commutation on the periphery is the composite of all the minimal commutation loops. It expresses the composition law of \eqref{eqn:enrch_fnctr_time:2}. This completes the proof of the "if" part. 
	
	The "only if" part also follows a study of the commutation diagrams \eqref{eqn:enrch:1:1} and \eqref{eqn:enrch:1:2} above. In this case, the smaller commutation loops which are assumed and which need to be proved switch roles. This completes the proof of Lemma \ref{lem:enrch:1}. \qed
	
	Lemma \ref{lem:enrch:1} thus establishes an equivalence of left $\TimeB$ actions and enriched functors. An left action can be created not only out of an enriched semigroup, but also from a monoid object. We next see that the latter also leads to an enriched semigroup action :    
	
	\begin{lemma} \label{lem:enriched:2}
		Let $\Context$ be a closed monoidal category and $\paran{ \TimeB, \add: \TimeB\otimes \TimeB \to \TimeB, \start :\Context\to \TimeB }$ be a monoid of $\Context$. If $F: \TimeB \otimes \Omega \to \Omega$ is a left $\TimeB$-action, then the following construction defines a functor $\Time_{\TimeB} \to \Context$ enriched over $\Context$.
		\begin{itemize}
			\item The object $*$ corresponds to $\Omega$.
			\item  A morphism $a: 1_{\Context} \to \TimeB$ in $\Time_{\TimeB}$ corresponds to
			\[
			\Omega \to 1_{\Context} \otimes \Omega \xrightarrow{a \otimes \Omega} \TimeB \otimes \Omega \xrightarrow{F} \Omega.
			\]
		\end{itemize}
	\end{lemma}
	
	We end with an observation that an enriched functor can be reconstructed by the construction of Lemma \ref{lem:enriched:2}  respectively.
	
	\begin{lemma} \label{lem:enriched:6}
		Let $\Context$ be a symmetric closed monoidal category and $\mathcal{T}$ be a one-object category enriched over $\Context$. Assume that a functor $F: \mathcal{T} \to \Context$ is an enriched functor. Then, $F$ coincides with the enriched functor obtained from a consecutive application of Lemma \ref{lem:enrch:1} and Lemma \ref{lem:enriched:2} to $F$.
	\end{lemma}
	
	%........................................................................................................................
	\subsection{Enriched functors from semigroups} \label{sec:enrich_time}
	
	The purpose of this short section is to review the conditions for being enriched functors and enriched natural transformations, when the the domain is a semigroup $\Time$ enriched over the codomain $\Context$. Thus according to the notation of Section \ref{sec:enrich_fnctr} this means  $\calY = \calM = \Context$ and $\calX = \Time$. In that case \eqref{eqn:def:enrich_nat:1} has only one morphism
	\begin{equation} \label{eqn:def:enrich_nat:Time:1}
		\omega_* : 1_{\Context} \to \IntHom{\Omega}{\Omega'}
	\end{equation}
	and \eqref{eqn:def:enrich_nat:2} becomes
	\begin{equation} \label{eqn:def:enrich_nat:Time:2}
		\begin{tikzcd}
			\TimeB \otimes 1_{\calC} \arrow{rr}{ G^\sharp \otimes \omega_* } && \Endo(\Omega') \otimes \IntHom{\Omega}{\Omega'} \arrow{d}{ \circ } \\
			\TimeB \arrow{u}{\cong} \arrow{d}[swap]{\cong} && \IntHom{\Omega}{\Omega'} \\
			1_{\calC} \otimes \TimeB \arrow{rr}[swap]{ \tilde{\omega}_{*} \otimes F^\sharp } && \IntHom{\Omega}{\Omega'} \otimes \Endo(\Omega) \arrow{u}[swap]{ \circ }
		\end{tikzcd}
	\end{equation}
	%	
	%........................................................................................................................
	\subsection{Proof of Theorem \ref{thm:1}} \label{sec:proof:thm:1}
	
	To prove the result, we shall utilize Lemma \ref{lem:enrch:1}. We have to show that $\Transfer{\Phi}$ satisfies the analog of \eqref{eqn:enrch_fnctr_time:2}, which are the following two commutative diagrams :
	\[\begin{tikzcd}
		1_{\Context} \otimes \IntHom{ \TimeB }{ \Omega } \arrow{dr}[swap]{ \start_{\star} \otimes \IntHom{ \TimeB }{ \Omega } } & \IntHom{ \TimeB }{ \Omega } \arrow[l, "\lambda_{[T;\Omega]}^{(L)-1}"'] \\
		&  \TimeB \otimes \IntHom{ \TimeB }{ \Omega } \arrow{u}[swap]{ \shift }
	\end{tikzcd} , \quad 
	\begin{tikzcd}
		\TimeB \otimes \TimeB \otimes \IntHom{ \TimeB }{ \Omega } \arrow{d}[swap]{ \add \otimes \IntHom{ \TimeB }{ \Omega } } \arrow{rr}{\TimeB \otimes \shift } && \TimeB \otimes \IntHom{ \TimeB }{ \Omega } \arrow{d}{ \shift } \\
		\TimeB \otimes \IntHom{ \TimeB }{ \Omega } \arrow{rr}{ \shift } && \IntHom{ \TimeB }{ \Omega }
	\end{tikzcd}\]
	We first verify the unit law. Let $\lambda^{(L)}$ and $\lambda^{(R)}$ denote the isomorphisms created by multiplication with $1_{\Context}$ on the left and right respectively. Then we have the following diagram :
	\[\begin{tikzcd} [column sep = large]
		\TimeB \otimes \IntHom{ \TimeB }{ \Omega } \arrow[bend right=30]{drr}[swap]{\cong} \arrow{rr}{ \TimeB \otimes \lambda^{(R)-1}_{ [T; \Omega] } } && \TimeB \otimes 1_{\Context} \otimes \IntHom{ \TimeB }{ \Omega } \arrow{d}[swap]{ \lambda^{(L)}_{ [T; \Omega]} \otimes \IntHom{\TimeB}{\Omega} } \arrow{rr}{ \TimeB \otimes \start_* \otimes \IntHom{ \TimeB }{ \Omega } } && \TimeB \otimes \TimeB \otimes \IntHom{ \TimeB }{ \Omega } \arrow{dll}{ \add \otimes \IntHom{ \TimeB }{ \Omega } } \arrow{d}[swap]{ \TimeB \otimes \shift } \\
		&& \TimeB \otimes \IntHom{ \TimeB }{ \Omega } \arrow{drr}[swap]{\eval_\Omega} && \TimeB \otimes \IntHom{ \TimeB }{ \Omega } \arrow{d}{\eval_\Omega} \\
		&& && \Omega
	\end{tikzcd}\]
	The commuting square has already been derived in \eqref{eqn:idl3v:2}. The two commuting triangles are trivial consequences of the unitor morphisms. Now note that the clockwise (CW) path from $\TimeB \otimes 1_{\Context} \otimes \IntHom{ \TimeB }{ \Omega }$ to $\Omega$ is the left-adjoint of the counter-clockwise (CCW) path in the unit rule. This path is isomorphic to the CCW path from $\TimeB \otimes \IntHom{ \TimeB }{ \Omega }$ to $\Omega$, which is a left-adjoint to $\lambda_{[T;\Omega]}^{(L)-1}$. Thus we have proved the commutation of the left-adjoints of the paths in the unit rule.
	
	We next prove the second diagram. We shall show that the left adjoints to the two paths in the diagram are equal. These left adjoints have already been determined in \eqref{eqn:proof:thm:1:1} and \eqref{eqn:proof:thm:1:5} respectively. We shall need two identities for the proof. The first is a trivial application of the associativity of $\add$ :
	\begin{equation} \label{eqn:proof:thm:1:2}
		\begin{tikzcd}
			\TimeB \otimes \TimeB \otimes \TimeB \otimes \IntHom{ \TimeB }{ \Omega } \arrow{d}[swap]{ \TimeB \otimes \add \otimes \IntHom{ \TimeB }{ \Omega } } \arrow{rr}{ \add \otimes \TimeB \otimes \IntHom{ \TimeB }{ \Omega } } && \TimeB \otimes \TimeB \otimes \IntHom{ \TimeB }{ \Omega } \arrow{d}{ \add \otimes \TimeB \otimes  \IntHom{ \TimeB }{ \Omega } } \\
			\TimeB \otimes \TimeB \otimes \IntHom{ \TimeB }{ \Omega } \arrow{rr}{ \TimeB \otimes \shift } && \TimeB \otimes \IntHom{ \TimeB }{ \Omega }
		\end{tikzcd}
	\end{equation}
	The second is a trivial application of the bifunctoriality of $\otimes$ :
	\begin{equation} \label{eqn:proof:thm:1:3}
		\begin{tikzcd}
			\TimeB \otimes \TimeB \otimes \TimeB \otimes  \IntHom{ \TimeB }{ \Omega } \arrow{d}[swap]{ \add \otimes\TimeB \otimes \IntHom{ \TimeB }{ \Omega } } \arrow{rr}{ \TimeB \otimes \TimeB \otimes \shift } && \TimeB \otimes \TimeB \otimes \IntHom{ \TimeB }{ \Omega } \arrow{d}{ \add \otimes \IntHom{ \TimeB }{ \Omega } } \\
			\TimeB \otimes \TimeB \otimes \IntHom{ \TimeB }{ \Omega } \arrow{rr}{ \TimeB \otimes \shift } && \TimeB \otimes \IntHom{ \TimeB }{ \Omega }
		\end{tikzcd}
	\end{equation}
	Three repeated applications of \eqref{eqn:idl3v:2} gives :
	\[\begin{tikzcd} [scale cd = 0.7]
		&& && \TimeB \otimes \TimeB \otimes \IntHom{ \TimeB }{ \Omega } \arrow{d}[swap]{ \add \otimes \IntHom{ \TimeB }{ \Omega } } \arrow{rr}{ \TimeB \otimes \shift } && \TimeB \otimes \IntHom{ \TimeB }{ \Omega } \arrow{d}{\eval_{\Omega}} \\
		&& \TimeB \otimes \TimeB \otimes \IntHom{ \TimeB }{ \Omega } \arrow{d}[swap]{ \add \otimes \IntHom{ \TimeB }{ \Omega } } \arrow{rr}{ \TimeB \otimes \shift } && \TimeB \otimes \IntHom{ \TimeB }{ \Omega } \arrow{rr}{\eval_{\Omega}} && \Omega \\
		\TimeB \otimes \TimeB \otimes \IntHom{ \TimeB }{ \Omega } \arrow{d}[swap]{ \TimeB \otimes \shift } \arrow{rr}{ \add \otimes \IntHom{ \TimeB }{ \Omega } } && \TimeB \otimes \IntHom{ \TimeB }{ \Omega } \arrow[bend right=20]{urrrr}[swap]{\eval_{\Omega}} \\
		\TimeB \otimes \IntHom{ \TimeB }{ \Omega } \arrow[bend right=40]{uurrrrrr}[swap]{\eval_{\Omega}}
	\end{tikzcd}\]
	Adding the commutations of \eqref{eqn:proof:thm:1:2} and \eqref{eqn:proof:thm:1:3} gives :
	\begin{equation} \label{eqn:proof:thm:1:4}
		\begin{tikzcd} [scale cd = 0.7]
			\TimeB \otimes \TimeB \otimes \TimeB \otimes \IntHom{ \TimeB }{ \Omega } \arrow{dd}[swap]{ \TimeB \otimes \add \otimes \IntHom{ \TimeB }{ \Omega } } \arrow{drr}{ \add \otimes \TimeB \otimes \IntHom{ \TimeB }{ \Omega } } \arrow{rrrr}{ \TimeB \otimes \TimeB \otimes \shift } && && \TimeB \otimes \TimeB \otimes \IntHom{ \TimeB }{ \Omega } \arrow{d}[swap]{ \add \otimes \IntHom{ \TimeB }{ \Omega } } \arrow{rr}{ \TimeB \otimes \shift } && \TimeB \otimes \IntHom{ \TimeB }{ \Omega } \arrow{d}{\eval_{\Omega}} \\
			&& \TimeB \otimes \TimeB \otimes \IntHom{ \TimeB }{ \Omega } \arrow{d}[swap]{ \add \otimes \IntHom{ \TimeB }{ \Omega } } \arrow{rr}{ \TimeB \otimes \shift } && \TimeB \otimes \IntHom{ \TimeB }{ \Omega } \arrow{rr}{\eval_{\Omega}} && \Omega \\
			\TimeB \otimes \TimeB \otimes \IntHom{ \TimeB }{ \Omega } \arrow{d}[swap]{ \TimeB \otimes \shift } \arrow{rr}{ \add \otimes \IntHom{ \TimeB }{ \Omega } } && \TimeB \otimes \IntHom{ \TimeB }{ \Omega } \arrow[bend right=20]{urrrr}[swap]{\eval_{\Omega}} \\
			\TimeB \otimes \IntHom{ \TimeB }{ \Omega } \arrow[bend right=40]{uurrrrrr}[swap]{\eval_{\Omega}}
		\end{tikzcd}
	\end{equation}
	The resulting commutation in the peripheral loop of \eqref{eqn:proof:thm:1:4} confirms the equality of \eqref{eqn:proof:thm:1:1} and \eqref{eqn:proof:thm:1:5}. This completes the proof of  \ref{thm:1}. \qed 
	%........................................................................................................................
	\subsection{Proof of Theorem \ref{thm:8}} \label{sec:proof:thm:8}
	
	The verification of compositionality and associativity in the two categories in claim (i) is routine and will be ommitted. The bijection between the objects claimed in statement (iii) have already been proved in Lemma \ref{lem:enrch:1}. Statements (iv) and (v) follow directly from (i), (ii) and (iii). Thus it remains to prove Claim (ii). This requires establishing a correspondence between morphisms in these two categories, that also preserves composition. 
	
	We will set up a functor $F:\FlowCat{\Context} \to \EnrichSGActCat{\Context}$ by
	\begin{itemize}
		\item  $\Phi \mapsto \tilde{\Phi}$ on objects.
		\item $(h:\Phi \to \Psi) \mapsto (h^\sharp:\tilde{\Phi} \to \tilde{\Psi})$ on morphisms.
	\end{itemize}
	
	To verify $F$ is well-defined, we need the following lemma, which implies that a morphism $h:\Omega \to \Omega'$ in $\Context$ is a morphism $h:\Phi \to \Psi$ in $\FlowCat{\Context}$ if and only if $h^\sharp: 1_\mathcal{C} \to \IntHom{\Omega}{\Omega'}$ defines a morphism $\tilde{\Phi} \to \tilde{\Psi}$ in $\EnrichSGActCat{\Context}$:
	
	\begin{lemma}
		Let $\Phi : \TimeB \otimes \Omega \to \Omega$ and $\Phi':\TimeB \otimes \Omega' \to \Omega'$ be morphisms in a context category $\mathcal{C}$. If $h:\Omega \to \Omega'$ and $h^\sharp: 1_\mathcal{C} \to \IntHom{\Omega}{\Omega'}$ are adjoints, then the morphism
		\[
		\TimeB \otimes \Omega \xrightarrow[]{\Phi} \Omega \xrightarrow[]{h} \Omega'
		\]
		is adjoint to 
		\[
		\TimeB \simeq 1_{\mathcal{C}}\otimes \TimeB \xrightarrow[]{h^\sharp \otimes\Phi^\sharp} \IntHom{\Omega}{\Omega'}\otimes \Endo{\Omega} \xrightarrow[]{\circ} \IntHom{\Omega}{\Omega}.
		\]
		Similarly, the morphism
		\[
		\TimeB \otimes \Omega \xrightarrow[]{\TimeB \otimes h} \TimeB \otimes \Omega' \xrightarrow[]{\Phi'} \Omega'
		\]
		is adjoint to 
		\[
		\TimeB \simeq \TimeB \otimes 1_{\mathcal{C}}\xrightarrow[]{\Phi'^\sharp\otimes h^\sharp} \Endo{\Omega'}\otimes\IntHom{\Omega}{\Omega'}  \xrightarrow[]{\circ} \IntHom{\Omega}{\Omega'}.
		\]
	\end{lemma}	
	By the commutation in \eqref{eqn:EnrichSGActCat:compose}, we have
	\[
	F(h'\circ h) = F(h') \circ F(h).
	\]
	Thus, $F$ is a functor and it is clear that $F$ has inverse.
	This completes the proof of Theorem \ref{thm:8}. \qed
	%........................................................................................................................
	\subsection{Proof of Theorem \ref{thm:2}} \label{sec:proof:thm:2}
	
	To prove the result, we shall utilize Lemma \ref{lem:enrch:1}. We have to show that $\Transfer{\Phi}$ satisfies the analog of \eqref{eqn:enrch_fnctr_time:2}. We need to show that the following diagrams commute :
	\[\begin{tikzcd}
		1_{\Context} \otimes \IntHom{ \TimeB }{ \Omega } \arrow{dr}[swap]{ \start_{\IntHom{ \TimeB }{ \Omega }} \otimes \IntHom{ \TimeB }{ \Omega } } & \TimeB \arrow{l}{ \lambda } \\
		& \TimeB \otimes \IntHom{ \TimeB }{ \Omega } \arrow{u}{ \Transfer{\Phi} }
	\end{tikzcd} , \quad \begin{tikzcd}
		\TimeB \otimes \TimeB \otimes \IntHom{ \TimeB }{ \Omega } \arrow{d}[swap]{ \add \otimes \IntHom{ \TimeB }{ \Omega } } \arrow{rr}{\TimeB \otimes \Transfer{\Phi} } && \TimeB \otimes \IntHom{ \TimeB }{ \Omega } \arrow{d}{ \Transfer{\Phi} } \\
		\TimeB \otimes \IntHom{ \TimeB }{ \Omega } \arrow{rr}{ \Transfer{\Phi} } && \IntHom{ \TimeB }{ \Omega }
	\end{tikzcd}\]
	We begin with the second diagram. We shall show that the adjoint to the two paths in the diagram are equal. By another use of Lemma \ref{lem:idl3v:1} and \eqref{eqn:idl3v:2}, these adjoints are :
	
	\begin{equation} \label{eqn:proof:thm:2:1}
		\begin{aligned}
			&\TimeB \otimes \TimeB \otimes \TimeB \otimes \IntHom{\TimeB}{\Omega} \xrightarrow{\TimeB \otimes \TimeB \otimes \Transfer{\Phi}}\TimeB \otimes \TimeB \otimes \IntHom{\TimeB}{\Omega} \xrightarrow{ \TimeB \otimes \Transfer{\Phi}}\TimeB \otimes \IntHom{\TimeB}{\Omega} \xrightarrow{\eval_{\Omega}} \Omega\\
			&\TimeB \otimes \TimeB \otimes \TimeB \otimes \IntHom{\TimeB}{\Omega} \xrightarrow{\TimeB \otimes \mu \otimes \IntHom{\TimeB}{\Omega}}\TimeB \otimes \TimeB \otimes \IntHom{\TimeB}{\Omega} \xrightarrow{ \TimeB \otimes \Transfer{\Phi}}\TimeB \otimes \IntHom{\TimeB}{\Omega} \xrightarrow{\eval_{\Omega}} \Omega
		\end{aligned}
	\end{equation}
	We shall need the following three identities in the proof : 
	\begin{equation} \label{eqn:proof:thm:2:5}
		\begin{tikzcd}
			\TimeB \otimes \TimeB \otimes \TimeB \otimes \IntHom{\TimeB}{\Omega} \arrow{rrrr}{ \TimeB \otimes \TimeB \otimes \Transfer{\Phi} } \arrow{d}[swap]{ \TimeB \otimes \TimeB \otimes \eval_{\Omega} } &&&& \TimeB \otimes \TimeB \otimes \IntHom{\TimeB}{\Omega} \arrow{d}{ \swap \otimes \IntHom{\TimeB}{\Omega} } \\
			\TimeB \otimes \TimeB \otimes \Omega \arrow{rr}{ \TimeB \otimes \Phi^\sharp } && \TimeB \otimes \Omega && \TimeB \otimes \TimeB \otimes \IntHom{\TimeB}{\Omega} \arrow{ll}[swap]{ \TimeB \otimes \eval_{\Omega} } 
		\end{tikzcd}
	\end{equation}
	\begin{equation} \label{eqn:proof:thm:2:3}
		\begin{tikzcd}
			\TimeB \otimes \TimeB \otimes \TimeB \otimes \IntHom{\TimeB}{\Omega} \arrow{d}[swap]{ \TimeB \otimes \add \otimes \IntHom{\TimeB}{\Omega} } \arrow{rrrr}{ \TimeB \otimes \TimeB \otimes \eval_{\Omega} } &&&& \TimeB \otimes \TimeB \otimes \Omega \arrow{d}{\add \otimes \Omega} \\
			\TimeB \otimes \TimeB \otimes \IntHom{\TimeB}{\Omega} \arrow{rr}[swap]{ \swap \otimes \IntHom{\TimeB}{\Omega} } && \TimeB \otimes \TimeB \otimes \IntHom{\TimeB}{\Omega} \arrow{rr}[swap]{ \TimeB \otimes \eval_{\Omega} } && \TimeB \otimes \Omega
		\end{tikzcd}
	\end{equation}
	\begin{equation} \label{eqn:proof:thm:2:4}
		\begin{tikzcd}
			\TimeB \otimes \TimeB \otimes \IntHom{\TimeB}{\Omega} \arrow{d}[swap]{ \TimeB \otimes \Transfer{\Phi} } \arrow{rr}{ \swap \otimes \IntHom{\TimeB}{\Omega} } && \TimeB \otimes \TimeB \otimes \IntHom{\TimeB}{\Omega} \arrow{d}{ \TimeB \otimes \eval_{\Omega} } \\
			\TimeB \otimes \IntHom{\TimeB}{\Omega} \arrow{r}[swap]{\eval_{\Omega}} & \Omega & \TimeB \otimes \Omega \arrow{l}{ \Phi }
		\end{tikzcd}
	\end{equation}

	The commutations in \eqref{eqn:idl3v:2}, \eqref{eqn:idl3v:3} and \eqref{eqn:proof:thm:2:5} combine to give
	\[\begin{tikzcd} [scale cd = 0.8, column sep = small]
		\TimeB \otimes \TimeB \otimes \TimeB \otimes \IntHom{\TimeB}{\Omega} \arrow{rrrrr}{ \TimeB \otimes \TimeB \otimes \Transfer{\Phi} } \arrow{dr}[swap]{ \TimeB \otimes \TimeB \otimes \eval_{\Omega} } &&&&& \TimeB \otimes \TimeB \otimes \IntHom{\TimeB}{\Omega} \arrow{d}[swap]{ \swap \otimes \IntHom{\TimeB}{\Omega} } \arrow[bend left=10]{dr}{ \TimeB \otimes \Transfer{\Phi} } \\
		& \TimeB \otimes \TimeB \otimes \Omega \arrow{d}{ \add \otimes \Omega } \arrow{rr}{ \TimeB \otimes \Phi } && \TimeB \otimes \Omega \arrow[bend right=5]{drrr}{\Phi} && \TimeB \otimes \TimeB \otimes \IntHom{\TimeB}{\Omega} \arrow{ll}[swap]{ \TimeB \otimes \eval_{\Omega} } & \TimeB \otimes \IntHom{\TimeB}{\Omega} \arrow{d}{\eval_{\Omega}} \\
		& \TimeB \otimes \Omega \arrow{rrrrr}[pos=0.3]{ \Phi } && && & \Omega 
	\end{tikzcd}\]
	The commutation in \eqref{eqn:proof:thm:2:3} can now be added to get
	\[\begin{tikzcd} [scale cd = 0.8, column sep = small]
		\TimeB \otimes \TimeB \otimes \TimeB \otimes \IntHom{\TimeB}{\Omega} \arrow[bend right=10]{dddd}[swap]{ \TimeB \otimes \add \otimes \IntHom{\TimeB}{\Omega} } \arrow{rrrrr}{ \TimeB \otimes \TimeB \otimes \Transfer{\Phi} } \arrow{dr}[swap]{ \TimeB \otimes \TimeB \otimes \eval_{\Omega} } &&&&& \TimeB \otimes \TimeB \otimes \IntHom{\TimeB}{\Omega} \arrow{d}[swap]{ \swap \otimes \IntHom{\TimeB}{\Omega} } \arrow[bend left=10]{dr}{ \TimeB \otimes \Transfer{\Phi} } \\
		& \TimeB \otimes \TimeB \otimes \Omega \arrow{d}{ \add \otimes \Omega } \arrow{rr}{ \TimeB \otimes \Phi } && \TimeB \otimes \Omega \arrow[bend right=5]{drrr}{\Phi} && \TimeB \otimes \TimeB \otimes \IntHom{\TimeB}{\Omega} \arrow{ll}[swap]{ \TimeB \otimes \eval_{\Omega} } & \TimeB \otimes \IntHom{\TimeB}{\Omega} \arrow{d}{\eval_{\Omega}} \\
		& \TimeB \otimes \Omega \arrow{rrrrr}[pos=0.3]{ \Phi } && && & \Omega \\
		& \TimeB \otimes \TimeB \otimes \IntHom{\TimeB}{\Omega} \arrow{u}[swap]{ \TimeB \otimes \eval_{\Omega} } \\
		\TimeB \otimes \TimeB \otimes \IntHom{\TimeB}{\Omega} \arrow{ur}[swap]{ \swap \otimes \IntHom{\TimeB}{\Omega} }
	\end{tikzcd}\]
	The commutation in \eqref{eqn:proof:thm:2:4} can now be added to get  the large diagram
	\begin{equation} \label{eqn:proof:thm:2:2}
		\begin{tikzcd} [scale cd = 0.8, column sep = small]
			\TimeB \otimes \TimeB \otimes \TimeB \otimes \IntHom{\TimeB}{\Omega} \arrow[bend right=10]{dddd}[swap]{ \TimeB \otimes \add \otimes \IntHom{\TimeB}{\Omega} } \arrow{rrrrr}{ \TimeB \otimes \TimeB \otimes \Transfer{\Phi} } \arrow{dr}[swap]{ \TimeB \otimes \TimeB \otimes \eval_{\Omega} } &&&&& \TimeB \otimes \TimeB \otimes \IntHom{\TimeB}{\Omega} \arrow{d}[swap]{ \swap \otimes \IntHom{\TimeB}{\Omega} } \arrow[bend left=10]{dr}{ \TimeB \otimes \Transfer{\Phi} } \\
			& \TimeB \otimes \TimeB \otimes \Omega \arrow{d}{ \add \otimes \Omega } \arrow{rr}{ \TimeB \otimes \Phi } && \TimeB \otimes \Omega \arrow[bend right=5]{drrr}{\Phi} && \TimeB \otimes \TimeB \otimes \IntHom{\TimeB}{\Omega} \arrow{ll}[swap]{ \TimeB \otimes \eval_{\Omega} } & \TimeB \otimes \IntHom{\TimeB}{\Omega} \arrow{d}{\eval_{\Omega}} \\
			& \TimeB \otimes \Omega \arrow{rrrrr}[pos=0.3]{ \Phi } && && & \Omega \\
			& \TimeB \otimes \TimeB \otimes \IntHom{\TimeB}{\Omega} \arrow{u}[swap]{ \TimeB \otimes \eval_{\Omega} } \\
			\TimeB \otimes \TimeB \otimes \IntHom{\TimeB}{\Omega} \arrow{ur}[swap]{ \swap \otimes \IntHom{\TimeB}{\Omega} } \arrow{rrr}{ \TimeB \otimes \Transfer{\Phi} } &&& \TimeB \otimes \IntHom{\TimeB}{\Omega} \arrow{uurrr}[swap]{ \eval_{\Omega} }
		\end{tikzcd}
	\end{equation}
	Note that the outer periphery of this diagram reflects the equality of the two paths in \eqref{eqn:proof:thm:2:1}.
	
	As for the unit law, we have the following diagram:
	\[\begin{tikzcd} [column sep = large, scale cd = 0.7]
		& {\TimeB \otimes 1_{\Context} \otimes \IntHom{ \TimeB }{ \Omega }} \arrow[lddd, "\TimeB \otimes \lambda"', bend right=49] \arrow[d, "{\swap \otimes \IntHom{ \TimeB }{ \Omega }}"] \arrow[rr, "{\TimeB \otimes \start \otimes \IntHom{ \TimeB }{ \Omega }}"] & & \TimeB \otimes \TimeB \otimes \IntHom{ \TimeB }{ \Omega } \arrow[d, "{\swap \otimes \IntHom{ \TimeB }{ \Omega }}"] \arrow[rr, "\TimeB \otimes \Transfer{\Phi}"] & & {\TimeB \otimes \IntHom{ \TimeB }{ \Omega }} \arrow[ddd, "\eval_{\Omega}"] \\
		& {1_{\Context} \otimes \TimeB \otimes \IntHom{ \TimeB }{ \Omega }} \arrow[ldd, "{\lambda \otimes \IntHom{ \TimeB }{ \Omega }}"'] \arrow[d, "1_{\Context} \otimes \eval_{\Omega}"] \arrow[rr, "{\start \otimes \TimeB \otimes \IntHom{ \TimeB }{ \Omega }}"] & & {\TimeB \otimes \TimeB \otimes \IntHom{ \TimeB }{ \Omega }} \arrow[d, "\TimeB \otimes \eval_{\Omega}"] & & \\
		& 1_{\Context} \otimes \Omega \arrow[rr, "\start \otimes \Omega"] \arrow[rrrrd, "\lambda"'] & & \TimeB \otimes \Omega \arrow[rrd, "\Phi"] & & \\
		{\TimeB \otimes \IntHom{ \TimeB }{ \Omega }} \arrow[rrrrr, "\eval_{\Omega}"] & & & & & \Omega 
	\end{tikzcd}\]
	The outermost clockwise path from $\TimeB \otimes \TimeB \otimes \IntHom{ \TimeB }{ \Omega }$ to $\Omega$ is the left adjoint of the transfer operator $\Transfer{\Phi}$ by Lemma \ref{lem:idl3v:1}. The counter clockwise path from $\TimeB \otimes \TimeB \otimes \IntHom{ \TimeB }{ \Omega }$ to $\Omega$ is the also left adjoint of $\Transfer{\Phi}$ from the definition \eqref{eqn:def:transfer}. This explains the top-right commuting square. All the other commuting triangles and squares follow directly from the consistency properties of $\start$ and $\swap$. Now note two peripheral paths in this diagram are exactly the left adjoints of the two paths in the unit law identity, again by Lemma \ref{lem:idl3v:1}. This completes the proof of Theorem \ref{thm:2}. \qed 
	%........................................................................................................................
	\subsection{Proof of Theorem \ref{thm:eqlz_dyn}} \label{sec:proof:eqlz_dyn}
	
	Our first observation is that the following diagram commutes.
	\begin{equation} \label{eqn:pxdi0}
		\begin{tikzcd}
			\TimeB \otimes E \arrow{rr}{\TimeB \otimes q} && \TimeB \otimes X \arrow[shift left=1]{rr}{F} \arrow[shift right=1]{rr}[swap]{G} && X
		\end{tikzcd}
	\end{equation}
	To see why we expand the above diagram into :
	\[\begin{tikzcd}
		\TimeB \otimes E \arrow{rr}{\TimeB \otimes q} && \TimeB \otimes X \arrow[shift left=1]{rr}{\TimeB \otimes F^\sharp} \arrow[shift right=1]{rr}[swap]{\TimeB \otimes G^\sharp} \arrow[bend right=20, shift left=1]{drr}{F} \arrow[bend right=20, shift right=1]{drr}[swap]{G} && \TimeB \otimes \IntHom{\TimeB}{X} \arrow{d}{\eval} \\
		&& && X
	\end{tikzcd}\]
	The top commutation follows from the equalizer property of $q$. The two commuting loops on the right follow from \eqref{eqn:idl3v:4}. This completes the proof of \eqref{eqn:pxdi0}.
	
	The next important observation is the following commutation
	\begin{equation} \label{eqn:sdlrx6}
		\begin{tikzcd}
			\TimeB \otimes E \arrow{rr}{\TimeB \otimes q} && \TimeB \otimes X \arrow[shift left=1]{r}{F} \arrow[shift right=1]{r}[swap]{G} & X \arrow[shift left=1]{r}{F^\sharp} \arrow[shift right=1]{r}[swap]{G^\sharp} & \IntHom{\TimeB}{X}
		\end{tikzcd}
	\end{equation}
	The first commutation is already covered in \eqref{eqn:pxdi0}. We thus need to show that the following two paths have the same adjoints :
	\[\begin{tikzcd}
		\TimeB \otimes E \arrow{rr}{\TimeB \otimes q} && \TimeB \otimes X \arrow{r}{F} & X \arrow{r}{F^\sharp} & \IntHom{\TimeB}{X} \\
		\TimeB \otimes E \arrow{rr}{\TimeB \otimes q} && \TimeB \otimes X \arrow{r}{G} & X \arrow{r}{G^\sharp} & \IntHom{\TimeB}{X}
	\end{tikzcd}\]
	Call these paths $\tilde{F}$ and $\tilde{G}$ respectively. Again by \eqref{eqn:idl3v:4}, it is equivalent to show that 
	\[ \eval \circ \paran{ \TimeB \otimes \tilde{F}} = \eval \circ \paran{ \TimeB \otimes \tilde{G}} . \]
	This follows from the following diagram
	\[\begin{tikzcd}
		\TimeB \otimes \TimeB \otimes E \arrow{d}{ \add \otimes E } \arrow{rr}{\TimeB \otimes \TimeB \otimes q} && \TimeB \otimes \TimeB \otimes X \arrow{d}{ \add \otimes X } \arrow{rr}{ \TimeB \otimes F} && \TimeB \otimes X \arrow{rr}{ \TimeB \otimes F^\sharp} \arrow[d, "F"] && \TimeB \otimes \IntHom{\TimeB}{X} \arrow[dll, "\eval"] \\
		\TimeB \otimes E \arrow{rr}{ \TimeB \otimes q} && \TimeB \otimes X \arrow[shift left=1]{rr}{F} \arrow[shift right=1]{rr}[swap]{G} && X  \\
		\TimeB \otimes \TimeB \otimes E \arrow{u}[swap]{ \add \otimes E } \arrow{rr}{\TimeB \otimes \TimeB \otimes q} && \TimeB \otimes \TimeB \otimes X \arrow{u}[swap]{ \add \otimes X } \arrow{rr}{ \TimeB \otimes G} && \TimeB \otimes X \arrow{rr}{ \TimeB \otimes G^\sharp} \arrow[u, "G"'] && \TimeB \otimes \IntHom{\TimeB}{X} \arrow[ull, "\eval"']
	\end{tikzcd}\]
	This completes the proof of \eqref{eqn:sdlrx6}.
	
	The commutation in \eqref{eqn:sdlrx6} along with the universality of the equalizer $q$ implies the existence of a unique morphism shown below as a dotted arrow
	\[\begin{tikzcd}
		\TimeB \otimes E \arrow[dotted, drr, "!"'] \arrow{rr}{\TimeB \otimes q} && \TimeB \otimes X \arrow[shift left=1]{r}{F} \arrow[shift right=1]{r}[swap]{G} & X \\
		&& E \arrow[bend right=20]{ur}[swap]{q}
	\end{tikzcd}\]
	This morphism is precisely the morphism $F\cap G$. Note that the commutation in \eqref{eqn:eqlz_dyn:3} is automatically borne by the figure above. 
	
	We next prove the second claim. To show that $F\cap G$ is a left action, we first verify its compositionality. 
	\[\begin{tikzcd}
		\TimeB \otimes \TimeB \otimes X \arrow[rrr, "\add \otimes X"] \arrow[ddd, shift left = 1, "\TimeB \otimes F"] \arrow[ddd, shift right = 1, "\TimeB \otimes G"'] & & & \TimeB \otimes X \arrow[ddd, shift left = 1, "F"] \arrow[ddd, shift right = 1, "G"'] \\
		& \TimeB \otimes \TimeB \otimes E \arrow[r, "\add \otimes E"] \arrow[ul, "\TimeB \otimes \TimeB \otimes q"'] \arrow[d, "\TimeB \otimes F\cap G"] & \TimeB \otimes E \arrow[dotted, d, "F\cap G"] \arrow[ur, "\TimeB \otimes q"] \\
		& \TimeB \otimes E \arrow[dl, "\TimeB \otimes q"] \arrow[r, "F\cap G"] & E \arrow[dr, "q"] \\
		\TimeB \otimes X \arrow[rrr, shift left = 1, "F"] \arrow[rrr, shift right = 1, "G"'] & & & X
	\end{tikzcd}\]
	In the diagram above all the commuting loops not involving the dotted arrow has already been established. The dotted arrow exists uniquely by the equalizer property of $q$. The unit law follows from the following diagram:
	\[\begin{tikzcd}
		E_{\Phi} \arrow[d, "q"'] \arrow[rr, "\lambda_{E_{\Phi}}^{(L)-1}"] && 1_{\Context} \otimes  E_{\Phi} \arrow[d, "1_{\Context} \otimes q"'] \arrow[rr, "\start_* \otimes \IntHom{ \TimeB }{ \Omega }"] && \TimeB \otimes E_{\Phi} \arrow[d, "\TimeB \otimes q"'] \arrow[rr, "\shiftF{\Phi}"] && E_{\Phi} \arrow[d, "q"'] \\
		\IntHom{ \TimeB }{ \Omega } \arrow[rr, "\lambda_{[T,\Omega]}^{(L)-1}"] && 1_{\Context} \otimes  \IntHom{ \TimeB }{ \Omega } \arrow[rr, "\start_* \otimes \IntHom{ \TimeB }{ \Omega }"] && \TimeB \otimes \IntHom{ \TimeB }{ \Omega } \arrow[rr, "\shift"] && \IntHom{ \TimeB }{ \Omega }
	\end{tikzcd}\]
	This completes the proof of Theorem \ref{thm:eqlz_dyn}. \qed
	
	%........................................................................................................................
	\subsection{Proof of Theorem \ref{thm:3}} \label{sec:proof:thm:3}
	
	The first four claims of the theorem follow directly from Theorem \ref{thm:eqlz_dyn}. It thus remains to prove Claim (v). We need two lemmas.
	
	\begin{lemma} \label{lem:iuhd9}
		The construction $\Transfer{}$ defines an endomorphism of the category of left $\TimeB$-actions.
	\end{lemma}
	
	\begin{proof} Let $f: \Phi \to \Psi$ be a morphism of left $\TimeB$-actions, with domains $\Omega, \Omega'$ respectively. We need to show the following commutation :
		\[\begin{tikzcd}
			\TimeB \otimes \IntHom{ \TimeB }{ \Omega } \arrow[d, "\Transfer{\Phi}"'] \arrow[rr, "\TimeB \otimes \IntHom{ \TimeB }{f}"] && \IntHom{ \TimeB }{ \Omega } \arrow[d, "\Transfer{\Phi'}"] \\
			\TimeB \otimes \IntHom{ \TimeB }{ \Omega' } \arrow[rr, "\IntHom{ \TimeB }{f}"] && \IntHom{ \TimeB }{ \Omega' }
		\end{tikzcd}\]
		The following commutation diagram establishes the equality of the left adjoints of these two morphisms :
		\[\begin{tikzcd} [scale cd = 0.8]
			\TimeB \otimes \TimeB \otimes \IntHom{ \TimeB }{ \Omega } \arrow[d, "\TimeB \otimes \TimeB \otimes \IntHom{ \TimeB }{f }"'] \arrow[rr, "{swap \otimes \IntHom{ \TimeB }{\Omega' }}"] \arrow[rrrrr, bend left=20, "\TimeB \otimes \Transfer{\Phi}"] & & {\TimeB \otimes \TimeB \otimes \IntHom{ \TimeB }{ \Omega }} \arrow[d, "\TimeB \otimes \TimeB \otimes \IntHom{ \TimeB }{f }"'] \arrow[rr, "\TimeB \otimes \eval_{\Omega}"] & & \TimeB \otimes \Omega \arrow[d, "\TimeB \otimes f"'] \arrow[rd, "\Phi"'] & \TimeB \otimes \IntHom{ \TimeB }{ \Omega } \arrow[d, "\eval_{\Omega}"] \arrow[ddr, bend left=30, "\TimeB \otimes \IntHom{ \TimeB }{f}"'] \\
			\TimeB \otimes \TimeB \otimes [\TimeB;\Omega'] \arrow[d, "\TimeB \otimes \Transfer{\Phi'}"] \arrow[rr, "{swap \otimes [\TimeB; \Omega']}"] & & {\TimeB \otimes \TimeB \otimes [\TimeB;\Omega']} \arrow[rr, "\TimeB \otimes \eval_{\Omega'}"] & & \TimeB \otimes \Omega' \arrow[d, "\tilde{\Phi'}"'] & \Omega \arrow[ld, "f"'] & \\
			{\TimeB \otimes [\TimeB;\Omega']} \arrow[rrrr, "\eval_{\Omega'}"] & & & & \Omega' && \TimeB \otimes \IntHom{ \TimeB }{ \Omega' } \arrow[ll, "\eval_{\Omega'}"'] 
		\end{tikzcd}\]
		The two paths along the periphery from $\TimeB \otimes \TimeB \otimes \IntHom{ \TimeB }{ \Omega }$ to $\Omega'$ are the left adjoints of the two paths in the commutative diagram to be proved. The commutation loop at the top is an application of Lemma \ref{lem:idl3v:1} and the definition \eqref{eqn:def:transfer} of the transfer operator. Ever other commutation follows from bi-functoriality.
	\end{proof}
	
	%\ref{lem:idl3v:4}, 	
	Now, the proof proceeds as follows. Lemmas \ref{lem:iuhd9} and \eqref{eqn:ijd9} together imply that the following morphisms have the same adjoint and therefore coincide:
	\[E_\Phi \xrightarrow{q} \IntHom{\TimeB}{\Omega} \xrightarrow{\IntHom{ \TimeB }{f}} \IntHom{\TimeB}{\Omega'} \xrightarrow[\Transfer{\Psi}^\flat]{\shift^\flat}\IntHom{\TimeB}{\IntHom{\TimeB}{\Omega'}}\]
	Then, the left commutative diagram in \eqref{eqn:thm:3:4} follows immediately from the universal property of the equalizer.
	
	For the second commutative diagram in \eqref{eqn:thm:3:4}, by a similar argument, we can also verify that the following morphisms coincide:
	\[\TimeB \otimes E_\Phi \xrightarrow{\TimeB \otimes E(f)} \TimeB \otimes E_\Psi \xrightarrow{\shiftF{\Psi}} E_\Psi \xrightarrow{q'}\IntHom{\TimeB}{\Omega'}\]
	
	\[\TimeB \otimes E_\Phi \xrightarrow{\shiftF{\Phi}}  E_\Phi \xrightarrow{E(f)} E_\Psi \xrightarrow{q'}\IntHom{\TimeB}{\Omega'}\]
	Again by the universal property of the equalizer, we obtain the desired diagram.
	This completes the proof of Theorem \ref{thm:3}. \qed 
	
	%........................................................................................................................
	\subsection{ Proof of Equation \eqref{eqn:ijd9}} \label{sec:proof:ijd9}
	
	In a manner analogous of the proof of Lemma \ref{lem:iuhd9} the proof follows from the following commutative diagram :
	\[\begin{tikzcd} [scale cd = 0.8]
		{\TimeB \otimes \IntHom{ \TimeB }{ \Omega' }} \arrow[rrrrrrr, "\eval_{\Omega'}", Shobuj] & & & & & & & \Omega' \\
		& \TimeB \otimes \IntHom{ \TimeB }{ \Omega' } \arrow[rrrrrru, "\eval_{\Omega'}"] & & \TimeB \otimes \IntHom{ \TimeB }{ \Omega } \arrow[ll, "\TimeB \otimes \IntHom{ \TimeB }{f}"] \arrow[rr, "\eval_{\Omega}"] & & \Omega \arrow[rru, "f"'] & & {} \\
		& {\TimeB \otimes \TimeB \otimes \IntHom{ \TimeB }{ \Omega' }} \arrow[u, "{\mu \otimes \IntHom{ \TimeB }{ \Omega' }}"'] \arrow[luu, "\TimeB \otimes \shift_{\Omega'}", bend left, Shobuj] & & \TimeB \otimes \TimeB \otimes \IntHom{ \TimeB }{ \Omega } \arrow[ll, "\TimeB \otimes \TimeB \otimes \IntHom{ \TimeB }{f}"', Shobuj] \arrow[rr, "\TimeB \otimes \shift_{\Omega}", Akashi] \arrow[u, "\mu \otimes \IntHom{ \TimeB }{ \Omega }"'] & & {\TimeB \otimes \IntHom{ \TimeB }{ \Omega }} \arrow[rr, "\TimeB \otimes \IntHom{ \TimeB }{f}", Akashi] \arrow[u, "\eval_{\Omega}"'] & & {\TimeB \otimes \IntHom{ \TimeB }{ \Omega' }} \arrow[uu, "\eval_{\Omega'}"', Akashi]
	\end{tikzcd}\]
	The rightmost commutation square follows from the naturality of the evaluation morphisms, the one in the middle follows from the flow property of $\shift$, and the one on the left follows from bi-functoriality of monoidal product. The commutation square at the top left corner is simply \eqref{eqn:idl3v:2}. The commutation square at the top with one vertex at $\Omega'$ also follows from the naturality of $\eval$.
	The CCW path in green, and the CW path in purple are the left adjoints to the CW and CCW paths respectively in \eqref{eqn:ijd9}, thus proving the commutation. \qed
	
	%........................................................................................................................
	\subsection{Proof of Theorem \ref{thm:6}} \label{sec:proof:thm:6}
	
	The proof proceeds in two steps. First we show that 
	\begin{equation} \label{eqn:thm:6:1}
		\shift^\flat \circ \Phi^\flat = \Transfer{\Phi}^\flat \circ \Phi^\flat
	\end{equation}
	We prove \eqref{eqn:thm:6:1} by proving that the double left adjoints (d.l.a) of the the two sides of the equation are equal. The d.l.a of the LHS of \eqref{eqn:thm:6:1} is the clockwise path shown below :
	\begin{equation} \label{eqn:thm:6:2}
		\begin{tikzcd}
			\TimeB \otimes \TimeB \otimes \Omega \arrow[rr, "T\otimes T\otimes \Phi^\flat"] \arrow[d, "\mu \otimes \Omega"] &  & {T\otimes T\otimes \IntHom{ \TimeB }{ \Omega }} \arrow[d, "{\mu \otimes \IntHom{ \TimeB }{ \Omega }}"] \\
			\TimeB \otimes \Omega \arrow[rr, "\TimeB \otimes \Phi^\flat"] \arrow[rrd, "\Phi"]                               &  & {\TimeB \otimes \IntHom{ \TimeB }{ \Omega }} \arrow[d, "\eval_\Omega"]                \\
			&  & \Omega 
		\end{tikzcd}
	\end{equation}
	The upper commutation square follows from the bi-functoriality of $\otimes$. The lower commutation triangle follows from \eqref{eqn:idl3v:1}. Overall, we get an equivalent expression for the d.l.a of the LHS of \eqref{eqn:thm:6:1}. We next study the d.l.a. of the RHS of \eqref{eqn:thm:6:1}. It is given by the clockwise path in the diagram below :
	\begin{equation} \label{eqn:thm:6:3}
		\begin{tikzcd}
			\TimeB \otimes \TimeB \otimes \Omega \arrow[rr, "T\otimes T\otimes \Phi^\flat"] \arrow[d, "\swap\otimes \Omega"] \arrow[dd, "\mu \otimes \Omega"', bend right = 50] &  & {T\otimes T\otimes \IntHom{ \TimeB }{ \Omega }} \arrow[d, "{\swap \otimes \IntHom{ \TimeB }{ \Omega }}"] \\
			\TimeB \otimes \TimeB \otimes \Omega \arrow[rr, "\TimeB \otimes T\otimes \Phi^\flat"] \arrow[rrd, "T \otimes\Phi"] \arrow[d, "\mu \otimes \Omega"] &  & {\TimeB \otimes T\otimes \IntHom{ \TimeB }{ \Omega }} \arrow[d, "\TimeB \otimes \eval_\Omega"]       \\
			T\otimes \Omega \arrow[rrd, "\Phi"] &  & T \otimes\Omega \arrow[d, "\Phi"] \\
			&  & \Omega       
		\end{tikzcd}
	\end{equation}
	Diagrams \eqref{eqn:thm:6:2} and \eqref{eqn:thm:6:3} establish the identity in \eqref{eqn:thm:6:1}. We have thus proven that $\Phi^\flat$ equalizes $\shift^\flat$ and $\Transfer{\Phi}^\flat$. We next prove that it is fact the equalizer. So we assume a commutation of the form
	\[\begin{tikzcd}
		A \arrow[r, "\alpha"] & {\IntHom{ \TimeB }{ \Omega }} \arrow[r, "\shift^\flat", bend left] \arrow[r, "\Transfer{\Phi}"', bend right] & {[T, \IntHom{ \TimeB }{ \Omega }]}
	\end{tikzcd}\]
	and show that $\alpha$ must factorize uniquely through $\Phi^\flat$. Define $\tilde{\alpha}: A \to \Omega$ as
	\[\begin{tikzcd}
		A \arrow[r, "\lambda"] & 1_{\Context} \otimes A \arrow[r, "\start_* \otimes A"] & \TimeB \otimes A \arrow[r, "\TimeB \otimes \alpha"] & {\TimeB \otimes \IntHom{ \TimeB }{ \Omega }} \arrow[r, "\eval_\Omega"] & \Omega
	\end{tikzcd}\]
	To verify that $\alpha = \Phi^\flat \circ \tilde{\alpha}$, we need the following commutation diagram :
	\[\begin{tikzcd}
		\TimeB \otimes \TimeB \otimes A \arrow[r, "T\otimes T\otimes \alpha"] & {\TimeB \otimes \TimeB \otimes \IntHom{ \TimeB }{ \Omega }} \arrow[rrr, "{(\TimeB \otimes \eval) \circ (\swap \otimes \IntHom{ \TimeB }{ \Omega })}"] \arrow[d, "{\mu \otimes \IntHom{ \TimeB }{ \Omega }}"] &&& \TimeB \otimes \Omega \arrow[d, "\Phi"] \\
		& {T\otimes \IntHom{ \TimeB }{ \Omega }} \arrow[rrr, "\eval_\Omega"] &&& \Omega 
	\end{tikzcd}\]
	Using this commutativity, we obtain the following diagram, which verifies the desired factorization of $\alpha$ by comparing their adjoints :
	\[\begin{tikzcd} [scale cd = 0.8, column sep = large]
		\TimeB \otimes A \arrow[r, "\lambda_{\TimeB}^{(R)-1} \otimes A", "\cong"'] \arrow[d, "\lambda_{\TimeB}^{(L)-1} \otimes A"', "\cong"] & \TimeB \otimes 1_{\Context} \otimes A \arrow[r, "\TimeB \otimes \start_* \otimes A"] \arrow[ld, "\swap \otimes A"', "\cong"] & \TimeB \otimes \TimeB \otimes A \arrow[r, "\TimeB \otimes \TimeB \otimes \alpha"] \arrow[ld, "\swap\otimes A"'] & {T\otimes T\otimes \IntHom{ \TimeB }{ \Omega }} \arrow[ld, "{\swap\otimes \IntHom{ \TimeB }{ \Omega }}"'] \arrow[rd, "\TimeB \otimes \eval_\Omega"] &\\
		1_{\Context} \otimes \TimeB \otimes A \arrow[r, "\start_* \otimes \TimeB \otimes A"] \arrow[rd, "\lambda_{\TimeB}^{(L)} \otimes A"', "\cong"] & \TimeB \otimes \TimeB \otimes A \arrow[r, "\TimeB \otimes \TimeB \otimes \alpha"] \arrow[d, "\mu \otimes A"] & {\TimeB \otimes \TimeB \otimes \IntHom{ \TimeB }{ \Omega }} \arrow[r, "{\swap \otimes \IntHom{ \TimeB }{ \Omega }}"] \arrow[d, "{\mu \otimes \IntHom{ \TimeB }{ \Omega }}"] & {\TimeB \otimes \TimeB \otimes \IntHom{ \TimeB }{ \Omega }} \arrow[r, "\TimeB \otimes \eval_\Omega"] & \TimeB \otimes \Omega \arrow[ld, "\Phi"] \arrow[d, "\TimeB \otimes \Phi^\flat"] \\
		& \TimeB \otimes A \arrow[r, "\TimeB \otimes \alpha"] & {\TimeB \otimes \IntHom{ \TimeB }{ \Omega }} \arrow[r, "\eval_\Omega"] & \Omega  &  \TimeB \otimes \IntHom{\TimeB}{\Omega} \arrow[l, "\eval_{\TimeB, \Omega}"]
	\end{tikzcd}\]
	Note that the the upper CCW path from $\TimeB \otimes A$ to $\Omega$ is the left adjoint of $\Phi^\flat \circ \tilde{\alpha}$, and the lower CW path between the same pair is the left adjoint of $\alpha$.
	
	We show the uniqueness of the factorization. Let $ \alpha = \Phi^\flat \circ \beta$ for some morphism $\beta$. Then, we have the following diagram:
	\[\begin{tikzcd}
		A \arrow[r, "\lambda_A^{(L)-1}"] \arrow[d, "\beta"] & 1_{\Context} \otimes A \arrow[r, "\start_* \otimes A"] \arrow[d, "1_{\Context} \otimes \beta"] & \TimeB \otimes A \arrow[r, "\TimeB \otimes \alpha"] \arrow[d, "\TimeB \otimes \beta"] & {\TimeB \otimes \IntHom{ \TimeB }{ \Omega }} \arrow[d, "\eval_\Omega"] \\
		\Omega \arrow[r, "\lambda_\Omega^{(L)-1}"] & 1_{\Context} \otimes \Omega \arrow[r, "\start_* \otimes \Omega"] \arrow[rr, "\lambda_\Omega^{(L)}"', bend right] & \TimeB \otimes \Omega \arrow[r, "\Phi"] & \Omega                     
	\end{tikzcd}\]
	Therefore, $\beta = \tilde{\alpha}$. This completes the proof of Theorem \ref{thm:6}. \qed  
	
	%........................................................................................................................
	\subsection{Proof of Theorem \ref{thm:sttnry}} \label{sec:proof:sttnry}
	
	Condition (i), which assumes that $\tilde{\omega}$ is an enriched natural functor is equivalent to the commutation relation \eqref{eqn:def:enrich_nat:Time:2}. The latter is simply \eqref{eqn:def:enrich_nat:cone:2} for the case $\calX = \Time$ and $\calY = \Context$. Now by \eqref{eqn:ibx9z3},  \eqref{eqn:def:enrich_nat:cone:2} is equivalent to \eqref{eqn:def:enrich_nat:cone:3}. Thus conditions (i), (ii) and (iii) are equivalent. We next show that this implies (iv). By Assumption \ref{A:trmnl} we have $\IntHom{a}{1} \cong 1$ for every $a\in \Context$. We prove \eqref{eqn:jd92} by establishing two identities.  Firstly :
	\begin{equation} \label{eqn:jd92:1}
		\begin{tikzcd}
			1_{\Context} \otimes \TimeB \arrow[dr, "\omega_* \otimes !"] \arrow[d, dashed] \\
			\IntHom{1_{\Context}}{\Omega} & \IntHom{1_{\Context}}{\Omega} \otimes \IntHom{1_{\Context}}{1_{\Context}} \arrow[l, "\circ"'] \\
		\end{tikzcd}
		\begin{tikzcd} {} \arrow[rr, "\mbox{right}", "\mbox{adjoint}"'] && {} \end{tikzcd}
		\begin{tikzcd}
			1_{\Context} \otimes \TimeB \otimes 1_{\Context} \arrow[d, dashed] \arrow[r, "\cong"] & \TimeB \otimes 1_{\Context} \arrow[d, "!"] \\
			\Omega & 1_{\Context} \arrow[l, "\omega"]
		\end{tikzcd}
	\end{equation}
	The dashed arrows in the two diagrams are thus left and right duals of each other. The next observation is 
	\begin{equation} \label{eqn:jd92:2}
		\begin{tikzcd} [scale cd = 0.7]
			1_{\Context} \otimes \TimeB \arrow[d, dashed] \arrow[r, "\lambda_T", "\cong"'] & \TimeB \arrow[r, " ( \rho_T )^{-1} ", "\cong"'] & \TimeB \otimes 1_{\Context} \arrow[dl, "\Phi^\sharp \otimes \tilde{\omega}"] \\
			\IntHom{1_{\Context}}{\Omega} & \IntHom{\Omega}{\Omega} \otimes \IntHom{1_{\Context}}{\Omega} \arrow[l, "\circ"']
		\end{tikzcd}
		\begin{tikzcd} {} \arrow[rr, "\mbox{right}", "\mbox{adjoint}"'] && {} \end{tikzcd}
		\begin{tikzcd} [scale cd = 0.7]
			1_{\Context} \otimes 1_{\Context} \otimes \TimeB \arrow[d, dashed] \arrow[rr, "1_{\Context} \otimes \swap_{1,T}"] && 1_{\Context} \otimes \TimeB \otimes 1_{\Context} \arrow[d, "\lambda_T \otimes 1_{\Context}"] \\
			\Omega & \TimeB \otimes \Omega \arrow[l, "\Phi"'] & \TimeB \otimes 1_{\Context} \arrow[l, "\TimeB \otimes \omega"']
		\end{tikzcd}
	\end{equation}
	Before proving \eqref{eqn:jd92:1} and \eqref{eqn:jd92:2} we note how it proves the claim of the theorem. The composite morphisms in the left diagrams of \eqref{eqn:jd92:1} and \eqref{eqn:jd92:2} are the two commuting morphisms in \eqref{eqn:def:enrich_nat:cone:3}. Thus they are equal. So their right adjoints are the same too. This means that the composite morphisms in the right diagrams of \eqref{eqn:jd92:1} and \eqref{eqn:jd92:2} are equal. But they are isomorphic to the two paths of \eqref{eqn:jd92}, thus making them commute. It remains to prove \eqref{eqn:jd92:1} and \eqref{eqn:jd92:2}. We prove them by proving commutation of their adjoints. %, taking the help of Lemma \ref{lem:idl3v:1}.
	
	The adjoints of the two paths in \eqref{eqn:jd92:1} appear on the periphery of the following  diagram : 
	\[\begin{tikzcd} [scale cd = 0.8]
		&&&& 1_{\Context} \otimes \IntHom{1_{\Context}}{\Omega} \otimes \IntHom{1_{\Context}}{1_{\Context}} \arrow[d, "1_{\Context} \otimes \circ"] \\
		1_{\Context} \otimes \TimeB \otimes 1_{\Context} \arrow[rr, "1_{\Context} \otimes ! \otimes \omega_*"] \arrow[rd, "1_{\Context} \otimes ! \otimes 1_{\Context}"] \arrow[ddd, "\lambda_{T} \otimes 1_{\Context}"] & & 1_{\Context} \otimes \IntHom{1_{\Context}}{1_{\Context}} \otimes \IntHom{1_{\Context}}{\Omega} \arrow[urr, "1_{\Context} \otimes \swap"] \arrow[d, "\eval_{1,1} \otimes \IntHom{1_{\Context}}{\Omega}"] && 1_{\Context} \otimes \IntHom{1_{\Context}}{\Omega} \arrow[d, "\eval_{1,\Omega}"] \\
		& 1_{\Context} \otimes 1_{\Context} \otimes \IntHom{1_{\Context}}{1_{\Context}} \arrow[rd, "1_{\Context} \otimes \eval_{1,1}"] & 1_{\Context} \otimes \IntHom{1_{\Context}}{\Omega} \arrow[rr, "\eval_{1,1}"] && \Omega \\
		& & 1_{\Context} \otimes 1_{\Context} \arrow[u, "1_{\Context} \otimes \omega_*"'] \arrow[rr, "\cong"] && 1_{\Context} \arrow[u, "\omega"] \\
		\TimeB \otimes 1_{\Context} \arrow[rru, "! \otimes 1_{\Context}"] \arrow[rrrru, bend right=10, "!\otimes 1_{\Context}"] & & & 
	\end{tikzcd}\]
	The adjoints of the two paths in \eqref{eqn:jd92:2} appear on the periphery of the following  diagram : 
	\[\begin{tikzcd} [scale cd = 0.8]
		1_{\Context} \otimes \TimeB \otimes 1_{\Context} \arrow[r] & \TimeB \otimes 1_{\Context} \arrow[r] \arrow[d] & 1_{\Context} \otimes 1_{\Context} \otimes \TimeB \arrow[rr, "1_{\Context} \otimes \omega_* \otimes \Phi^\sharp"] \arrow[ld] \arrow[d, "1_{\Context} \otimes \tilde{\omega} \otimes \TimeB"] & & 1_{\Context} \otimes \IntHom{1_{\Context}}{\Omega} \otimes \IntHom{\Omega}{\Omega} \arrow[r, "1_{\Context} \otimes \swap"] \arrow[d, "\eval_{1, \Omega} \otimes \IntHom{\Omega}{\Omega}"] & 1_{\Context} \otimes \IntHom{\Omega}{\Omega} \otimes \IntHom{1_{\Context}}{\Omega} \arrow[d, "1_{\Context} \otimes \circ"] \\
		& 1_{\Context} \otimes \TimeB \arrow[rrrd, bend right=10, "\omega \otimes \TimeB"] & 1_{\Context} \otimes \IntHom{1_{\Context}}{\Omega} \otimes \TimeB \arrow[rrd, "\eval_{1,\Omega} \otimes \TimeB"] & & \Omega \otimes \IntHom{\Omega}{\Omega} \arrow[dr, "\eval"] & 1_{\Context} \otimes \IntHom{1_{\Context}}{\Omega} \arrow[d, "\eval_{1, \Omega}"] \\
		& & & & \Omega \otimes \TimeB \arrow[u, "\Phi^\sharp \otimes \Omega"'] \arrow[r, "\Phi"'] & \Omega
	\end{tikzcd}\]
	This completes the proof of Theorem \ref{thm:sttnry}. \qed
	
	%-_-_-_-_-_-_-_-_-_-_-_-_-_-_-_-_-_-_-_-_-_-_-_-_-_-_-_-_-_-_-_-_-_-_-_-_-_-_-_-_-_-_-_-_-_-_-_-_-_-_-_-_-_-_-_-_-_-_-_-_-_-_-_-_-_-_-_-_-_-_-_-_-_-_-_-
	%\bibliographystyle{\Path unsrt_inline_url}	\bibliography{\Path References,Ref}

\begin{thebibliography}{10}

\bibitem{BehrischEtAl2017dyn}
M~Behrisch, S~Kerkhoff, R~Poschel, F~Schneider, and S~Siegmund.
\newblock \href{http://dx.doi.org/10.1007/s10485-015-9409-8}{Dynamical systems
  in categories}.
\newblock {\em Appl. Cat. Struct.}, 25:29--57, 2017.

\bibitem{fritz2020Stoch}
T.~Fritz.
\newblock \href{http://dx.doi.org/10.1016/j.aim.2020.107239}{A synthetic
  approach to {M}arkov kernels, conditional independence and theorems on
  sufficient statistics}.
\newblock {\em Adv. Math.}, 370:107239, 2020.

\bibitem{MossPerrone2022ergdc}
S.~Moss and P.~Perrone.
\newblock \href{http://dx.doi.org/10.1017/etds.2023.6}{A category-theoretic
  proof of the ergodic decomposition theorem}.
\newblock {\em Erg. Th. Dyn. Sys.}, 2022.

\bibitem{DasSuda2024recon}
S.~Das and T.~Suda.
\newblock \href{https://epubs.siam.org/doi/10.1137/25M1733203}{Dynamics, data
  and reconstruction}.
\newblock {\em Siam J. Appl. Dyn.}, 25, 2026.

\bibitem{Suda2022Poincare}
T.~Suda.
\newblock \href{http://dx.doi.org/10.1080/14689367.2022.2027346}{A categorical
  view of poincar{\'e} maps and suspension flows}.
\newblock {\em Dynamical Systems}, 37(1):159--179, 2022.

\bibitem{SchultzEtAl2020dyn}
P.~Schultz et~al.
\newblock \href{http://dx.doi.org/10.1007/s10485-019-09565-x}{Dynamical systems
  and sheaves}.
\newblock {\em Appl. Cat. Struct.}, 28(1):1--57, 2020.

\bibitem{VagnerSpivakLerman2015}
D.~Vagner, D.~Spivak, and E.~Lerman.
\newblock \href{http://dx.doi.org/10.48550/arXiv.1408.1598}{Algebras of open
  dynamical systems on the operad of wiring diagrams}.
\newblock {\em Theory Appl. Categories}, 30:1793--1822, 2015.

\bibitem{LibkindEtAl2021operadic}
S.~Libkind et~al.
\newblock \href{http://dx.doi.org/10.4204/EPTCS.372.14}{Operadic modeling of
  dynamical systems: mathematics and computation}, 2021.

\bibitem{FongEtAl2016open}
B.~Fong, P.~Soboci{\'n}ski, and P.~Rapisarda.
\newblock A categorical approach to open and interconnected dynamical systems.
\newblock In {\em Proceedings of the 31st annual ACM/IEEE symposium on Logic in
  Computer Science}, pages 495--504, 2016.

\bibitem{pareigis1977moduleA}
B.~Pareigis.
\newblock \href{http://dx.doi.org/10.5486/pmd.1977.24.1-2.24}{Non-additive ring
  and module theory 1}.
\newblock {\em Publicationes mathematicae}, 1977.

\bibitem{pareigis1977moduleB}
B.~Pareigis.
\newblock Non-additive ring and module theory 2.
\newblock {\em Publ. Math. Debrecen}, 24(3-4):351--361, 1977.

\bibitem{Suda2023dynamical}
T.~Suda.
\newblock \href{http://dx.doi.org/10.1016/j.jde.2024.01.043}{Dynamical
  properties in the axiomatic theory of ordinary differential equations}.
\newblock {\em J. Diff. Eq.}, 391:1--24, 2024.

\bibitem{Suda2023partial}
T.~Suda.
\newblock \href{http://dx.doi.org/10.3934/dcdsb.2023082}{On partial maps
  derived from flows.}
\newblock {\em Disc. Contin. Dyn. Syst.}, 28(11), 2023.

\bibitem{Lawvere1973metric}
W.~Lawvere.
\newblock \href{http://dx.doi.org/10.1007/BF02924844}{Metric spaces,
  generalized logic, and closed categories}.
\newblock {\em Rendiconti del seminario mat{\'e}matico e fisico di Milano},
  43:135--166, 1973.

\bibitem{Lawvere1986serious}
W.~Lawvere.
\newblock \href{http://dx.doi.org/10.15446/recolma}{Taking categories
  seriously}.
\newblock {\em Revista colombiana de matematicas}, 20(3-4):147--178, 1986.

\bibitem{Keller2007graded}
B.~Keller.
\newblock \href{http://dx.doi.org/10.4171/022-2/8}{On differential graded
  categories}.
\newblock \href{http://dx.doi.org/10.4171/022-2/8}{}, pages 151--190, 2007.

\bibitem{Bergner2007model}
J.~Bergner.
\newblock \href{http://dx.doi.org/10.1090/S0002-9947-06-03987-0}{A model
  category structure on the category of simplicial categories}.
\newblock {\em Trans. Amer. Math. Soc.}, 359(5):2043--2058, 2007.

\bibitem{day1970closed}
B.~Day.
\newblock \href{http://dx.doi.org/10.1007/BFb0060438}{On closed categories of
  functors}.
\newblock \href{http://dx.doi.org/10.1007/BFb0060438}{In {\em Reports of the
  Midwest Category Seminar IV}}, pages 1--38. Springer, 1970.

\bibitem{power2011unicity}
J.~Power.
\newblock \href{http://dx.doi.org/10.1007/s10485-009-9202-7}{Unicity of
  enrichment over cat or gpd}.
\newblock {\em Applied Categorical Structures}, 19(1):293--299, 2011.

\bibitem{McdermottUustalu2022makes}
D.~McDermott and T.~Uustalu.
\newblock \href{http://dx.doi.org/10.48550/arXiv.2207.00851}{What makes a
  strong monad?}, 2022.

\bibitem{Myers2021ecds}
D.~Myers.
\newblock Enriched categories as dynamical systems.
\newblock Technical report, Topos Institute, 2021.

\bibitem{raissy2010torus}
J.~Raissy.
\newblock \href{http://dx.doi.org/10.1007/s12220-009-9108-5}{Torus actions in
  the normalization problem}.
\newblock {\em Journal of Geometric Analysis}, 20:472--524, 2010.

\bibitem{Das2023Koop_susp}
S.~Das.
\newblock
  \href{www.sciencedirect.com/science/article/abs/pii/S0022039625008137}{Smooth
  {K}oopman eigenfunctions}.
\newblock {\em J. Diff. Eq.}, 453(1), 2026.

\bibitem{DSSY_Mes_QuasiP_2016}
S.~Das et~al.
\newblock \href{http://dx.doi.org/10.1209/0295-5075/114/40005}{Measuring
  quasiperiodicity}.
\newblock {\em Europhys. Lett. EPL}, 114:40005--40012, 2016.

\bibitem{DasJim2017chaos}
S.~Das and J.~Yorke.
\newblock \href{http://dx.doi.org/10.1137/17M1113199}{Multichaos from
  quasiperiodicity}.
\newblock {\em SIAM J. Appl. Dyn. Syst.}, 16(4):2196–2212, 2017.

\bibitem{AlexanderEtAl1992}
J.~Alexander, J.~Yorke, Z.~You, and I.~Kan.
\newblock \href{http://dx.doi.org/10.1142/S0218127492000446}{Riddled basins}.
\newblock {\em Int. J. Bifur. Chaos}, 2(04):795--813, 1992.

\bibitem{newhouse2004new}
S.~Newhouse.
\newblock \href{http://dx.doi.org/10.1017/S0143385703000749}{New phenomena
  associated with homoclinic tangencies}.
\newblock {\em Erg. Th. Dyn. Sys.}, 24(5):1725--1738, 2004.

\bibitem{LiaoEtAl2018}
G.~Liao, W.~Sun, E.~Vargas, and S.~Wang.
\newblock \href{http://dx.doi.org/10.1017/etds.2018.33}{Approximation of
  {B}ernoulli measures for non-uniformly hyperbolic systems}.
\newblock {\em Erg. Th. Dyn. Sys.}, 22(5):1--15, 2018.

\bibitem{Das2023CatEntropy}
S.~Das.
\newblock \href{https://arxiv.org/pdf/2301.09205.pdf}{The categorical basis of
  dynamical entropy}.
\newblock {\em Applied Categorical Structures}, 32, 2024.

\bibitem{Leinster2020integration}
T.~Leinster.
\newblock
  \href{https://www.research.ed.ac.uk/en/publications/the-categorical-origins-of-lebesgue-integration}{The
  categorical origins of {L}ebesgue integration}.
\newblock {\em J. London Math. Soc.}, 2023.

\bibitem{ClementinoTholen1997sep}
M.~Clementino and W.~Tholen.
\newblock \href{http://dx.doi.org/10.1016/S0166-8641(96)00087-9}{Separation
  versus connectedness}.
\newblock {\em Topology Appl.}, 75(2):143--181, 1997.

\bibitem{ClementinoGiuliTholen1996}
M.~Clementino, E.~Giuli, and W.~Tholen.
\newblock Topology in a category: compactness.
\newblock {\em Portugaliae Math.}, 53(4):397--434, 1996.

\bibitem{Leinster2022eventual}
T.~Leinster.
\newblock \href{http://dx.doi.org/10.48550/arXiv.2210.00302}{The eventual
  image}, 2022.

\bibitem{BuragoEtAl2001metric}
D.~Burago et~al.
\newblock \href{http://dx.doi.org/10.1090/gsm/033}{{\em A course in metric
  geometry}}, volume~33.
\newblock American Mathematical Society Providence, 2001.

\bibitem{Doob1953book}
M.~Doob.
\newblock \href{www.wiley.com/en-us/Stochastic+Processes-p-9780471523697}{{\em
  Stochastic processes}}.
\newblock Wiley, 1953.

\bibitem{Kelly1982enriched}
G.~Kelly and M.~Kelly.
\newblock \href{http://tac.mta.ca/tac/reprints/articles/10/tr10.pdf}{{\em Basic
  concepts of enriched category theory}}, volume~64.
\newblock CUP Archive, 1982.

\bibitem{Flori2013Topos}
C.~Flori.
\newblock \href{http://dx.doi.org/10.1007/978-3-642-35713-8}{{\em A First
  Course in Topos Quantum Theory}}.
\newblock Lecture Notes in Physics. Springer, 2013.

\bibitem{Barr1986}
M.~Barr.
\newblock \href{http://dx.doi.org/10.4153/CMB-1986-079-9}{Fuzzy set theory and
  topos theory}.
\newblock {\em Canad. Math. Bull.}, 29(4):501--508, 1986.

\bibitem{Steenrod1967cnvnnt}
N.~Steenrod.
\newblock \href{http://dx.doi.org/10.1307/mmj/1028999711}{A convenient category
  of topological spaces}.
\newblock {\em Michigan Math. J.}, 14(2):133--152, 1967.

\bibitem{HeunenEtAl2017cnvnt}
C.~Heunen et~al.
\newblock \href{http://dx.doi.org/10.1109/LICS.2017.8005137}{A convenient
  category for higher-order probability theory}.
\newblock \href{http://dx.doi.org/10.1109/LICS.2017.8005137}{In {\em 2017 32nd
  Annual ACM/IEEE Symposium on Logic in Computer Science (LICS)}}, pages 1--12.
  IEEE, 2017.

\bibitem{Stacey2008smooth}
A.~Stacey.
\newblock \href{http://dx.doi.org/10.48550/arXiv.0802.2225}{Comparative
  smootheology}, 2008.

\bibitem{BaezHoffnung2011cnvnnt}
J.~Baez and A.~Hoffnung.
\newblock \href{http://dx.doi.org/10.1090/S0002-9947-2011-05107-X}{Convenient
  categories of smooth spaces}.
\newblock {\em Trans. Amer. Math. Soc.}, 363(11):5789--5825, 2011.

\bibitem{Chen1977iterated}
Kuo-Tsai Chen.
\newblock \href{http://dx.doi.org/10.1090/S0002-9904-1977-14320-6}{Iterated
  path integrals}.
\newblock {\em Bull. Amer. Math. Soc.}, 83(5):831--879, 1977.

\bibitem{FritzRischel2020inf}
T.~Fritz and E.~Rischel.
\newblock \href{http://dx.doi.org/10.32408/compositionality-2-3}{Infinite
  products and zero-one laws in categorical probability}.
\newblock {\em Compositionality}, 2, 2020.

\bibitem{vanBelle2023mrtngl}
R.~Van Belle.
\newblock \href{http://dx.doi.org/10.48550/arXiv.2305.03421}{A categorical
  treatment of the radon-nikodym theorem and martingales}, 2023.

\bibitem{AndrekaNemeti1982direct}
H.~Andr{\'e}ka and I.~N{\'e}meti.
\newblock \href{http://dx.doi.org/10.4064/-9-1-75-88}{Direct limits and
  filtered colimits are strongly equivalent in all categories}.
\newblock {\em Banach Center Publications}, 9:75--88, 1982.

\end{thebibliography}

\end{document}